\documentclass{amsart}

\usepackage{amssymb,amsmath}
\usepackage[mathscr]{euscript}

\newcommand{\RR}{\mathscr{R}} 
\newcommand{\PP}{\mathscr{P}} 
\newcommand{\Q}{\mathscr{Q}} 
\newcommand{\p}{\varpi} 
\newcommand{\ratk}{\mathit{k}} 
\newcommand{\resk}{\mathfrak{f}}  
\newcommand{\charp}{p} 
\newcommand{\resp}{p_{\resk}}  
\newcommand{\algK}{\mathrm{K}} 
\newcommand{\algG}{\mathbb{G}} 
\newcommand{\algM}{\mathbb{M}}
\newcommand{\ratM}{\mathrm{M}}
\newcommand{\Spn}{{\mathrm{Sp}_{2n}}}
\newcommand{\SLn}{{\mathrm{SL}_n}}
\newcommand{\GLn}{{\mathrm{GL}_n}}
\newcommand{\GLnn}{{\mathrm{GL}_{2n}}}
\newcommand{\ratG}{\mathrm{G}} 
\newcommand{\algg}{{\mathfrak{g}_\algK}} 
\newcommand{\ratg}{\mathfrak{g}} 

\newcommand{\buil}{\mathcal{B}} 
\newcommand{\apart}{\mathcal{A}} 
\newcommand{\ratT}{\mathrm{T}} 
\newcommand{\ratB}{\mathrm{B}} 
\newcommand{\algT}{\mathbb{T}} 
\newcommand{\ratt}{\mathfrak{t}} 
\newcommand{\facet}{{\mathcal{F}}} 
\newcommand{\nilpk}{\mathrm{Nil}(\ratk)} 

\newcommand{\Ad}{\mathrm{Ad}}
\newcommand{\rank}{\mathrm{rank\;}}

\newcommand{\Z}{\mathbb{Z}}   
\newcommand{\ep}{\varepsilon} 
\newcommand{\val}{\textrm{val}} 
\newcommand{\sltwo}{\mathfrak{sl}_2}

\newcommand{\Hasse}{\textrm{Hasse}}
\newcommand{\diag}{\textrm{diag}}
\newcommand{\Qani}{Q_{\textrm{aniso}}}

\theoremstyle{plain}
\newtheorem{theorem}{Theorem}[section]

\newtheorem{lemma}[theorem]{Lemma}
\newtheorem{proposition}[theorem]{Proposition}
\newtheorem{corollary}[theorem]{Corollary}

\theoremstyle{definition}
\newtheorem{definition}[theorem]{Definition}
\newtheorem{remark}[theorem]{Remark}

\theoremstyle{remark}
\newtheorem{example}[theorem]{Example}

\numberwithin{equation}{section}
\numberwithin{table}{section}

\begin{document}

\title[Nilpotent Orbits]{On Nilpotent Orbits of $\SLn$ and $\Spn$ over a Local Non-Archimedean Field}
\author{Monica Nevins}
\address{Department of Mathematics and Statistics, University of Ottawa, K1N 6N5}
\email{mnevins@uottawa.ca} 



\date{\today}

 \thanks{This research supported by Faculty of Science at the University of Ottawa, as well as by the Natural Sciences and Engineering Research Council of Canada (NSERC).}

\begin{abstract}
We relate the partition-type parametrization of rational (arithmetic)
nilpotent adjoint orbits of the classical groups $\SLn$ and
$\Spn$ over local non-Archimedean fields 
with a parametrization, introduced by DeBacker in 2002, which uses the
associated Bruhat-Tits building to relate the question to one
over the residue field.
\end{abstract}

\maketitle

\section{Introduction}

Let $\ratk$ be a local non-Archimedean field and let $\algG$ be a reductive
linear algebraic group defined over $\ratk$.
In \cite{DeBacker}, DeBacker parametrizes the set of $\ratk$-rational 
(that is, arithmetic)
nilpotent adjoint orbits of
$\algG$
by equivalence classes of objects coming from the Bruhat-Tits
building of the group.  This parametrization 
forms a key step in
DeBacker's proof of the range of validity of the Harish-Chandra-Howe 
character expansion in \cite{DeBackerhomog}.

As a special case of the Dynkin-Kostant classification, one can parametrize the algebraic (that is, geometric)  
nilpotent adjoint orbits of a classical 
algebraic group $\algG$ explicitly 
by way of 
the action of $\sltwo$-triples
on the standard representation $V$, and this classification can be 
conveniently interpreted via partitions of $n = \dim(V)$.  
The
$\ratk$-rational points of each algebraic orbit decompose into one or more
rational orbits under the action of $\algG(\ratk)$.  The parametrization
of these rational orbits thus additionally involves terms dependent
on the field $\ratk$, such as equivalence classes of nondegenerate
quadratic forms.

In this paper, we give such a partition-type classification of
rational nilpotent orbits of $\SLn$ in Proposition~\ref{P:slorbits}
and of $\Spn$ in Propositions~\ref{P:sp} and \ref{P:representatives},
based on the argument for real groups given in \cite{CMcG}.
We then interpret DeBacker's 
parameter set explicitly for the groups $\SLn$ and $\Spn$,
and define a map from this partition-type parametrization to the
DeBacker one for rational nilpotent orbits of these groups.
This is the content of Theorem~\ref{T:slcorrespondence} and 
Theorem~\ref{T:debackerspn}. 

One also may classify algebraic nilpotent adjoint orbits 
of reductive linear algebraic groups
via the Bala-Carter classification, and DeBacker describes
his pa\-ra\-met\-ri\-za\-tion as ``an affine analogue of Bala-Carter
theory'' \cite{DeBacker}.  A real analogue of the Bala-Carter
classification was provided in \cite{noel} by No\"el.
Though originally proven over algebraically closed fields
of characteristic either zero or sufficiently large, Bala-Carter theory
has been extended to algebraically 
closed fields of good characteristic by Pommerening.    Recent
work of McNinch, including \cite{McN1,McNinch}, has also advanced the
theory for rational orbits over fields of low characteristic.
More generally and classically, given an algebraic orbit
one may apply Galois cohomology towards understanding the
corresponding rational orbits; this was used previously by the author
in studying admissible nilpotent orbits of exceptional $p$-adic
groups in \cite{Nevins}.  Waldspurger gave a parametrization
of rational nilpotent orbits of symplectic, orthogonal and unitary
groups over the $p$-adic numbers in \cite[I.6]{Waldspurger};
his symplectic case is equivalent to our Proposition~\ref{P:sp}
for such $\ratk$.

These various parametrization schemes have different ranges
of validity and of applicability.  
The Bala-Carter parametrization of algebraic orbits, as well as
the DeBacker parametrization of rational orbits, 
applies to all reductive linear algebraic groups, whereas
partition-type classifications are restricted to the classical
groups.  The Bala-Carter classification avoids the
problematic (in
low characteristic) use of Lie triples upon which the
DeBacker parametrization in \cite{DeBacker} 
and partition-type classifications
rely.  Partition-type classifications yield explicit
representatives and encode much information about the corresponding
algebraic orbit, including dimension, the closure ordering on
orbits, and whether the orbit is special,
even or distinguished.  In contrast, we see here that these properties are not 
easily read from the DeBacker parametrization.

There are a number of interesting questions to pursue with regards
to DeBacker's parametrization.  
The first is that the parametrization varies with the choice of
a real number $r$.  For different choices of $r$, the number of
classes of objects in the building ($r$-facets) to which orbits
are associated can increase, and so correspondingly the number of
orbits associated to each $r$-facet will decrease.   In particular, 
two orbits associated to the same class of $r$-facet 
may be associated to different classes of $s$-facets, for $r\neq s$.
That said, given the $r$-facet to which
an orbit is associated, plus additional information
about the orbit (a Lie triple and adapted one-parameter subgroup)
it is easy to find the $s$-facet to which it is associated, for any $s$
(Corollary~\ref{C:justone}).  Choosing
$r$ to be irrational maximizes the number of classes of $r$-facets,
a feature we exploit in the proofs of Theorems~\ref{T:slcorrespondence} and 
\ref{T:debackerspn}. 

Secondly, the classification of the 
classes of $r$-facets seems quite difficult in general.  We determine some
equivalence classes for the case of $\SLn$ in Corollary~\ref{C:assoc}.
The problem can be reduced to a finite computation, since it suffices
to consider all $r$-facets meeting a fundamental chamber, but an
elegant solution seems elusive.

Thirdly, the dimension of the $r$-facet
to which a given rational orbit is associated is not in general an invariant
of the algebraic orbit.   This feature may perhaps offer more tools to distinguish between the various rational orbits in one algebraic class.   We prove the 
dimension is an invariant for $\SLn$ in 
Corollary~\ref{C:dimsl} and is not for
$\Spn$ in Corollary~\ref{C:dim}.

Finally, the DeBacker parametrization is proven under hypotheses
which require large residual characteristic, although DeBacker 
conjectures \cite{DeBacker} that the correspondence should
hold more generally.  We explore this question through some
examples in Section~\ref{S:smallreschar}, where we see that
one sticking point is that the intrinsic definition of the classes
arising in the correspondence (distinguished pairs, Definition~\ref{D:distinguished}) doesn't extend to small residual characteristic.  We can
also show that (not surprisingly) the correspondence cannot hold
if the residual characteristic is not good for $\algG$.

This paper is organized as follows.  In Section~\ref{S:Notation} we set our notation and recall several well-known
results about Lie ($\sltwo$) triples, the partition-type classification of
algebraic nilpotent orbits, Bruhat-Tits buildings and Moy-Prasad filtrations.  In Section~\ref{S:debacker}, we summarize the key results needed here about
DeBacker's parametrization
of rational nilpotent orbits.

In Section~\ref{S:sl} we turn our attention to the group $\SLn$.  We first
give the parame\-tri\-zation of rational
nilpotent orbits of $\SLn(\ratk)$ in Proposition~\ref{P:slorbits}; this  
is presumably well-known.  Our purpose is to produce
preferred representatives of the orbits, which we use to
deduce the corresponding DeBacker parameters
in Theorem~\ref{T:slcorrespondence}.  We conclude the section
by stating several Corollaries of the main theorem.

In Section~\ref{S:sp}, we consider the group $\Spn$.  
We begin by describing the partition-type classification of
rational nilpotent orbits in Section~\ref{SS:sppartition}.  To give
explicit representatives of these orbits, 
we briefly recall some facts about quadratic forms over local fields,
and then choose preferred representatives
for equivalence classes of nondegenerate quadratic forms, 
in
Section~\ref{SS:lam}.  The corresponding  orbit representatives are given
in Section~\ref{SS:sporbits}, and our main theorem is presented in
Section~\ref{SS:spdebacker}.  

We conclude in Section~\ref{S:smallreschar} with some illustrative 
examples and discussion about issues arising in small residual characteristic.
  
\subsection*{Thanks}
I wish to thank Stephen DeBacker for
some discussions on the subject, and Jason Levy for spending many hours 
learning about buildings with me.

\section{Preliminaries} \label{S:Notation}

Let $\ratk$ be a non-Archimedean local field of characteristic $\charp$
with finite residue field $\resk$ of characteristic $\resp >0$.
Then either $\ratk$ is a $p$-adic
field of characteristic $0$, or $\ratk$ is
a field of Laurent series over a finite field $\resk$, and $\charp = \resp > 0$.  
For the DeBacker correspondence we will require $\resp$ to be sufficiently
large, but for the partition-type classification it will suffice to 
ask that $\charp$ is either zero or sufficiently large.

Let $\RR$ denote the integer ring of $\ratk$ and $\PP$ the maximal
ideal of $\RR$.
Let $\p$ denote a uniformizer of $\ratk$ and let the discrete valuation
on $\ratk$ be normalized so that $\val(\p) = 1$.  
Let $\algK$ be an algebraic closure of $\ratk$.  

Let $\algG$ be a classical 
linear algebraic group defined over $\ratk$, and identify $\algG = \algG(\algK)$.  Write $\algG^0$ for the connected component of the identity of $\algG$.
Let $\algg$ denote its Lie algebra and
$V$ the vector space of its natural representation.
Set $\ratG = \algG(\ratk)$,  the group of $\ratk$-rational points of $\algG$,
and set $\ratg=\algg(\ratk)$.  We also write $V$ in place of $V(\ratk)$, where
this will not cause confusion.

 In this
paper, we will consider $\algG = \SLn$ or $\algG = \Spn$,
for some $n \geq 2$.  The group $\SLn$ consists of unimodular matrices.
Set $J = \left[ \begin{smallmatrix} 0 & I \\ -I & 0 \end{smallmatrix} \right]$,
where $I$ is the $n \times n$ identity matrix.  Then, writing $A^\dagger$ for
the transpose of the matrix  $A$, we embed
$\Spn(\algK)$ into the general linear group $\GLnn(\algK)$ as $\{ g \in \GLnn(\algK) \colon g^\dagger J g = J\}$.
We denote the symplectic form on $V$ corresponding to $J$ by $\langle x, y \rangle = x^\dagger J y$.

A \emph{partition} $\lambda$ of a positive integer
$n$ is a sequence $(\lambda_1, \cdots, \lambda_t)$ of positive integers in weakly decreasing order
with the property that $\sum_{i=1}^t \lambda_i = n$.  
The $\lambda_i$ are called the \emph{parts} of the partition $\lambda$
and for any $j$, the multiplicity of $j$ in $\lambda$, denoted $m_j^\lambda$ or $m_j$, is the number of parts
$\lambda_i$ such that $\lambda_i =j$.  
Write $\gcd(\lambda)$ for the greatest common divisor of
the parts of $\lambda$.

\subsection{Some $\sltwo(\ratk)$-modules} \label{S:sltwo}
  Consider the basis $\{e,h,f\}$ of $\sltwo(\ratk)$
for which $[h,e]=2e$, $[h,f]=-2f$, and $[e,f]=h$.  For each
positive integer $j$, we can construct a $j$-dimensional
module $W_j$ of $\sltwo(\ratk)$ with basis
$\{e^{j-1}w, e^{j-2}w, \cdots, ew, w\}$,
subject to
$h(e^iw) = (2i+1-j)\; e^i w$, $e^jw = 0$, $fw = 0$, and for $i>0$,
$f(e^iw) = i(j-i) \; e^{i-1}w$.
The $\ratk$-subspace
spanned by $w$ is called the lowest weight space of this module;
in general we refer to the eigenspaces of $h$ as weight spaces.
Write $\pi_j$ for the representation afforded by this module; then
with respect to this basis 
we have
$\pi_j(e) = J_j$, the upper triangular matrix
in Jordan normal form corresponding to a single Jordan block.  We also
have
\begin{align} \label{HrYr}
\pi_j(h) = H_j & = \diag(j-1, j-3, \cdots, -j+1),\\
\notag \pi_j(f) = Y_j & = J_j^\dagger\; \diag(j-1, 2(j-2), \cdots, j-1, 0),
\end{align}
which are diagonal and lower triangular matrices, respectively.

Now let $W$ be 
a finite-dimensional
$\sltwo(\ratk)$-module affording a representation $\pi$ such that there
exists a positive integer $m$ with  $\pi(e)^m=\pi(f)^m=0$.  Then
\cite[Thm 5.4.8]{Carter} if $\ratk$ has characteristic zero or
$\charp > m+1$, $W$ decomposes as a direct sum of irreducible
submodules, each one of which is isomorphic to $(\pi_j,W_j)$ for
some $1 \leq j \leq m$.  The dimensions of these irreducibles,
counted with multiplicity, define a partition $\lambda$ of $\dim(W)$.
Conversely, any partition of a positive integer $n$ 
(such that if $\charp >0$ then
$n+1<\charp$) defines a unique isomorphism class of
$\sltwo(\ratk)$-module.

In particular, choosing bases
for the irreducibles as above, the matrices for $\pi(e)$, $\pi(h)$ 
and $\pi(f)$ take block-diagonal form, with blocks of size 
equal to the parts of $\lambda$, in decreasing order.  
Write $J_\lambda$, $H_\lambda$ and $Y_\lambda$ for these
$n \times n$ matrices; so $J_\lambda$ is the upper triangular
matrix in Jordan normal form corresponding to the partition $\lambda$.

\subsection{Lie triples} \label{SS:lietriples}
Suppose $\ratk$ is a field of characteristic zero,
or of characteristic $\charp > 3(h-1)$ where $h$ is the Coxeter number of $\algG$.
(Recall that $h = n$ for $\algG = \SLn$ and $h=2n$ for $\algG = \Spn$.)  
Let $X$ be a nonzero nilpotent element in $\ratg$.   
Then by the Jacobson-Morozov
Theorem (see, for instance, \cite[\S5.4]{Carter}) there exists a Lie algebra
homomorphism $\phi \colon \sltwo \to \algg$
defined over $\ratk$
such that $\phi(e) = X$; we call $( \phi(f), \phi(h), \phi(e))$ 
(or, by abuse of notation, the map $\phi$)
a \emph{Lie triple} corresponding to $X$.  
 
Furthermore, there exists a homomorphism of algebraic groups $\varphi \colon \mathrm{SL}_2 \to \algG$,
defined over $\ratk$, such that $d\varphi = \phi$.  Define 
a one-parameter subgroup $\lambda$ of $\ratG$ via
$\lambda(t) = \varphi\left( \left[\begin{smallmatrix} t & 0 \\ 0 & t^{-1} \end{smallmatrix} \right] \right)$
for $t \in \ratk^\times$.  Then $d\lambda(1) = \phi(h)$ and $\lambda$
is said to be \emph{adapted} to the Lie triple $(\phi(f), \phi(h), \phi(e))$
\cite[Def 4.5.6]{DeBacker}.

\subsection{Classification of algebraic nilpotent orbits} \label{SS:classical}
Let $\ratk, \algK$ be as above.  The material in this section is adapted from \cite{CMcG} (where the hypothesis $\charp=0$ was assumed).

Let $X$ be a nonzero nilpotent element in $\algg$ and $\phi$ 
a corresponding Lie triple.
Through $\phi$, $V$ is a completely reducible  $\sltwo(\algK)$-module.
Let $V(j)$ denote
the isotypic component of $V$ of all $j$-dimensional irreducible
submodules; then there is an isomorphism
\begin{align*}
\phi_j \colon L(j) \otimes W_j& \to V(j)\\
v \otimes e^iw &\mapsto X^{i}v,
\end{align*}
where $L(j)$ denotes the subspace of lowest weight vectors of $V(j)$.
Then we may write
\begin{equation} \label{E:decomp}
V = \bigoplus_{j \in \mathbb{N}}V(j) \simeq \bigoplus_{j \in \mathbb{N}} L(j) \otimes W_j.
\end{equation}
If $X'$ is another nilpotent element in the same $\algG$-orbit, 
with associated Lie triple $\phi'$, then by the Jacobson-Morozov Theorem 
there exists an element $g \in \algG$
intertwining $\phi$ and $\phi'$ as representations of $\sltwo(\algK)$ (see, for example,
\cite[Proposition 5.6.4]{Carter}).   It follows that the 
modules \eqref{E:decomp} arising from $X$ and $X'$ are isomorphic, and
thus give rise to the same partition of $\dim(V)$.
In fact, we have the following
well-known result about nilpotent orbits over $\algK$;
see, for example, \cite[\S 3.1, 5.1]{CMcG}.  

\begin{proposition} \label{P:partitions}
When $\algG = \SLn$, the set of nilpotent adjoint orbits is 
in one-to-one correspondence with the set of partitions of $n$.
When $\algG= \Spn$, the set of nilpotent adjoint orbits is 
in one-to-one correspondence with the set of partitions of 
$2n$ in which all odd parts occur with even multiplicity.
\end{proposition}

When $\algG$ is quasi-split one deduces from \cite[Theorem 4.2]{Kottwitz}
or \cite{Dokovic} that
every nilpotent $\algG$-orbit which is defined over $\ratk$ contains
a $\ratk$-rational element.  Thus each algebraic nilpotent orbit 
of $\SLn$ and $\Spn$ decomposes into one or more rational
orbits.

More precisely, we may apply the same reasoning as above 
to deduce that nilpotent $\ratG$-orbits are in one-to-one correspondence with
$\ratG$-orbits of $\ratk$-rational Lie triples in $\ratg$; the proofs of the
Jacobson-Morozov theorem and Kostant's theorem in \cite[Chap 3]{CMcG} are
unchanged over $\ratk$.  However, two non-conjugate Lie triples $\phi$ and $\phi'$ 
may give rise
to the same partition $\lambda$;
this happens exactly when none of the
endomorphisms from $V$ to $V$ which intertwine the two representations $\phi$ and $\phi'$
of $\sltwo(\ratk)$ lie in the group $\ratG$.  Distinguishing these cases is the subject of the first parts of each of Sections~\ref{S:sl} and \ref{S:sp}.

\subsection{Buildings and Moy-Prasad filtrations}  
Let $\ratk$ be a field of zero or odd characteristic.  We consult
\cite{MoyPrasad} for the theory of Moy-Prasad filtrations.
Let $\ratG = \SLn(\ratk)$ or $\Spn(\ratk)$ so that $\algG$ is
connected and split over $\ratk$.  Let  $\buil(\ratG) = \buil(\algG,\ratk)$ 
denote
the Bruhat-Tits building of $\ratG$.  The building is endowed with a
$\ratG$-action.  For any $x \in \buil(\ratG)$, set $\ratG_x = \{ g \in \ratG \colon g \cdot x = x\}$.  Since $\algG$ is simple and simply-connected, this is a parahoric subgroup of $\ratG$.
Now let $\algT$ be a maximal 
torus of $\algG$ which is $\ratk$-split and set $\ratT = \algT(\ratk)$.
Let $\apart = \apart(\algT) \subset \buil(\ratG)$ be the corresponding
apartment.  Then $\apart$ is the affine space underlying
$X_*(\algT) \otimes_{\mathbb{Z}} \mathbb{R}$
where $X_*(\algT)$ is the group of $\ratk$-rational
cocharacters (one-parameter
subgroups) of $\algT$.   

Let $X^*(\algT)$ be the group of $\ratk$-rational
characters of $\algT$.  
Recall there is a natural pairing $X^*(\algT) \times X_*(\algT) \to \mathbb{Z}$, which we denote by $(\alpha, \mu) \mapsto m(\alpha,\mu)$.

Let $\Phi = \Phi(\algG,\algT) \subseteq X^*(\algT)$ be the set of roots of $\algT$
in $\algG$, and let $\Psi = \Psi(\algG,\algT)$ denote the set of affine
roots relative to $\Phi$ and the valuation on $\ratk$.  If we fix an origin, as above, then
we may write
$$
\Psi = \{ \phi + n \colon \phi \in \Phi, n \in \mathbb{Z} \}.
$$
We say $\phi = \dot{\psi}$ is the gradient of $\psi$.   Each element
of $\Psi$ defines a map $\psi \colon \apart \to \mathbb{R}$ by
$\psi(\lambda \otimes s) = \langle \phi + n, \lambda \otimes s \rangle = s m(\phi,\lambda) + n$. 

For $\alpha \in \Phi$, let $\ratg_\alpha$ denote the $\alpha$-root subspace 
of $\ratg$.
To each $\psi \in \Psi$, we associate an $\RR$-submodule $\ratg_\psi$ of the
corresponding root space $\ratg_{\dot{\psi}}$ by choosing an
isomorphism $\gamma \colon \ratk \to \ratg_{\dot{\psi}}$ such
that $\gamma(\RR) = \ratg_{\dot{\psi}} \cap \ratg(\RR)$ and
setting 
$$
\ratg_\psi = \{ X \in \ratg_{\dot{\psi}} \colon \val(\gamma^{-1}(X)) \geq \psi - \dot{\psi}\}.
$$

To each pair $(x,r) \in \buil(\ratG) \times \mathbb{R}$, Moy and Prasad
have associated an $\RR$-subalgebra of $\ratg$, denote $\ratg_{x,r}$, defined
as follows.  Choose an apartment $\apart = \apart(\algT)$ such that $x \in \apart$.
Set $\ratt$ to be the Lie algebra of $\ratT$; define  
$$
\ratt_r = \{ H \in \ratt \colon \val(d\chi(H)) \geq r \; \;\forall \chi \in X^*(\algT)\}.
$$
Then 
$$
\ratg_{x,r} = \ratt_r \oplus \sum_{\psi \in \Psi \colon \psi(x) \geq r} \ratg_{\psi}.
$$
Similarly, we define $\ratt_{r+}$ and in turn $\ratg_{x,r+}$ by replacing each 
inequality above with a strict inequality.  Then $\ratg_{x,r+} = \ratg_{x,s}$
for some $s > r$ depending on $x$.

Explicitly, for $\algG = \SLn$, consider the apartment $\apart$
corresponding to the diagonal torus $\algT$; then identifying
$\apart = X_*(\algT) \otimes_\mathbb{Z} \mathbb{R}$ we 
have
$$
\Phi = \{ e_i - e_j \colon 1 \leq i \neq j \leq n \}
$$
where $e_i(\diag(t^{x_1}, t^{x_2}, \ldots, t^{x_n})\otimes s) \doteq sx_i$.
In particular, for any $a \in \apart$, $\sum_{i=1}^n e_i(a) = 0$.   
Our choice of simple system is the
set of roots $\{ \alpha_i = e_i - e_{i+1} \colon 1 \leq i < n\}$.
Each root space is one-dimensional.  An 
element $X \in \ratg_{e_i-e_j}$,
given in matrix form, has $X_{kl} = 0$ unless $(k,l) = (i,j)$.

Similarly for $\algG = \Spn$ consider the apartment $\apart$
corresponding to the diagonal torus $\algT$.  By our embedding of 
$\Spn$ the elements of $\ratT$
are diagonal matrices of the form $\tau = \diag(t_1, t_2, \ldots, t_n, t_1^{-1}, t_2^{-1}, \ldots, t_n^{-1})$.
Then
working as above, we have
$$
\Phi = \left\{ e_i - e_j, \pm (e_i + e_j), \pm 2e_i \colon 1 \leq i \neq j \leq n \right\},
$$
where here $e_i(\diag(t^{x_1}, \ldots, t^{x_n}, t^{-x_1}, \ldots, t^{-x_n}) \otimes s) \doteq sx_i$.   
The one-dimensional root spaces may be identified in matrix form as follows.  
We have
\begin{align*}
\ratg_{e_i-e_j} &= \left\{ \left[ \begin{matrix} A & 0 \\ 0 & -A^\dagger \end{matrix} \right] \colon
A \in M_{n\times n}(\ratk), A_{kl} = 0 \text{\;unless\;} (k,l)=(i,j) \right\},\\
\ratg_{e_i+e_j} &= \left\{ \left[ \begin{matrix} 0 & A \\ 0 & 0 \end{matrix} \right] \colon
A \in M_{n\times n}(\ratk), A = A^\dagger, A_{kl} = 0 \text{\;unless\;} \{k,l\}=\{i,j\} \right\}, \\
\ratg_{2e_i} &= \left\{ \left[ \begin{matrix} 0 & A \\ 0 & 0 \end{matrix} \right] \colon
A \in M_{n\times n}(\ratk), A_{kl} = 0 \text{\;unless\;} k=l=i \right\}, 
\end{align*}
corresponding to the positive roots, 
and for each $\alpha \in \Phi^+$, $X \in \ratg_{-\alpha}$ if and only if $X^\dagger \in \ratg_{\alpha}$.

\section{DeBacker's parametrization of nilpotent orbits} \label{S:debacker}

All material in this section is summarized from \cite{DeBacker}.  
For the case of split simple groups, the hypotheses of \cite[\S 4.2]{DeBacker}
(under which Lemmas~\ref{L:DB1} and \ref{L:DB2}, and Theorem~\ref{T:debacker}, recalled below, are proven) are satisfied by taking the residual characteristic $\resp > 3(h-1)$.

Let $r \in \mathbb{R}$.  Fix $\apart$, $\Phi$, $\Psi$ as above.  
 For each $\psi \in \Psi$ define
$$
H_{\psi-r} = \{ x \in \apart \colon \psi(x) = r \}.
$$
Given a finite subset $S \subseteq \Psi$, define
$H_S = \bigcap_{\psi \in S} H_{\psi - r}$.
Then any connected component $\facet$ in $H_S$ of
the complement 
$$
H_S \setminus \bigcup_{\psi \in \Psi \setminus S} H_{\psi-r},
$$
for some $S$, is called an $r$-facet of $\apart$.
Denote the smallest affine subspace of $\apart$ containing $\facet$ by $A(\facet,\apart)$ and define $\dim(\facet) = \dim A(\facet,\apart)$.

An $r$-facet has the property that for each $x,y \in \facet$, 
\begin{equation} \label{facet}
\ratg_{x,r} = \ratg_{y,r} \quad \text{and} \quad \ratg_{x,r+} = \ratg_{y,r+}.
\end{equation}
More generally, for each $x \in \buil(\ratG)$, the \emph{generalized $r$-facet}
$\facet^\buil$ containing $x$ is the set of all $y \in \buil(\ratG)$ satisfying \eqref{facet}.  Then $\facet^\buil$ is an open convex subset of $\buil(\ratG)$
whose intersection with any apartment, when nonempty, is an $r$-facet; moreover
each such nonempty intersection has the same dimension.  Call two generalized
$r$-facets $\facet_1^\buil$ and $\facet_2^\buil$ \emph{$r$-associate} if there exists a $g \in \ratG$ and
an apartment $\apart$ such that 
$A(\facet_1^\buil\cap \apart,\apart) = A(g\facet_2^\buil\cap \apart,\apart) \neq \emptyset$.

Now for each $x \in \facet \subset \facet^\buil$, consider the quotient space
$$
V_{\facet^\buil} \doteq V_\facet \doteq V_{x,r} \doteq \ratg_{x,r}/\ratg_{x,r+};
$$
this is naturally a vector space over the residue field $\resk$ with
a $\resk$-linear action by the parahoric subgroup $\ratG_x$.
Note that $V_{x,r}$ is entirely defined by the roots $\psi \in \Psi$ such that $\psi(x)=r$,
which in turn only depends only on $A(\facet, \apart)$.  In fact there
is a natural identification of $V_{\facet_1}$ with $V_{\facet_2}$ whenever 
$\facet_1$ and $\facet_2$ are such that $A(\facet_1, \apart) = A(\facet_2, \apart)$.

Call an element $v \in V_{x,r}$ \emph{degenerate} if there exists
a nilpotent element $X \in \ratg_{x,r}$ mapping to $v$ under the quotient
map.  We consider the set 
$$
I_r^n \doteq \{ (\facet,v) \colon \text{$v$ is a degenerate element of $V_{\facet}$} \}
$$
(or use generalized $r$-facets in place of $r$-facets, as in \cite[Def 5.3.1]{DeBacker}).
Define an equivalence relation $\sim$ on $I_r^n$ via 
$(\facet_1, v_1) \sim (\facet_2,v_2)$ 
if there exists a $g \in \ratG$ such that $A(\facet_1, \apart) = A(g\facet_2, \apart)$
and such that,  under the natural identification of $V_{\facet_1}$ with
$Ad(g)V_{\facet_2}$, the elements $v_1$ and $Ad(g)v_2$ lie in the same 
orbit under $\ratG_x$ for any $x \in \facet_1$.

\begin{lemma} \cite[Corollary 5.2.4]{DeBacker} \label{L:DB1}
Let $(\facet,v) \in I_r^n$, $v \neq 0$.  Extend $v$ to a Lie triple
$$
(w,h,v) \in V_{x,-r} \times V_{x,0} \times V_{x,r}.
$$
Let
$$
(Y,H,X) \in \ratg_{x,-r} \times \ratg_{x,0} \times \ratg_{x,r}
$$
be any lift of this triple to a Lie triple over $\ratk$. 
Then
we have that $\mathcal{O}(\facet,v) \doteq \Ad(\ratG)X$ is the unique nilpotent orbit of
minimal dimension whose intersection with the coset 
$v = X + \ratg_{x,r+}$ is nontrivial.
\end{lemma}

If $v=0$, define $\mathcal{O}(\facet,v) = \{0\}$.  Let $\nilpk$ denote
the set of rational nilpotent orbits in $\ratg$.

\begin{lemma} \cite[Lemma 5.3.3]{DeBacker}
\label{L:DB2} The map
\begin{align*}
\gamma \colon I_r^n/\sim &\to \nilpk\\
(\facet,v) &\mapsto \mathcal{O}(\facet,v),
\end{align*}
is  well defined and surjective.
\end{lemma}

The map $\gamma$ is not bijective; see Example~\ref{exampledist} for
an interesting example.

Given a Lie triple $(Y,H,X)$ over $\ratk$, set
$$
\buil(Y,H,X) \doteq \buil_r(Y,H,X) \doteq \{ x \in \buil(\ratG) \colon Y \in \ratg_{x,-r}, X \in \ratg_{x,r} \}.
$$
This is a nonempty closed convex subset of $\buil(\ratG)$ which is
the union of generalized $r$-facets, such that any two generalized
$r$-facets of maximal dimension in $\buil(Y,H,X)$ are $r$-associate.
Moreover, given $(\facet,v)$ and an associated triple $(Y,H,X)$ as
in Lemma~\ref{L:DB1} we have that $\facet^\buil \subset \buil(Y,H,X)$.
 
\begin{definition} \label{D:distinguished}
The degenerate pair $(\facet,v)$ is \emph{distinguished} if 
$\facet^\buil$ is a generalized $r$-facet of maximal dimension
in $\buil(Y,H,X)$, where $(\facet,v)$ and $(Y,H,X)$ are related as in Lemma~\ref{L:DB1}.
\end{definition}

Denote the set of distinguished pairs by $I_r^d$, and the restriction
of $\gamma$ to $I_r^d/\sim$ by $\gamma_d$.

\begin{theorem} \cite[Theorem 5.6.1]{DeBacker} \label{T:debacker}
The map $\gamma_d \colon (I_r^d/\sim) \to \nilpk$ is a bijection.
\end{theorem}

The following results, which are implicit in \cite{DeBacker}, are key
to establishing our correspondence.

\begin{proposition} \label{P:mine}
The equivalence class of $(\facet,v)$ is distinguished if and 
only if $(\facet,v)$ is maximal, in terms of the dimension of
$\facet$, among all degenerate pairs occurring in $\gamma^{-1}(\mathcal{O}(\facet,v))$.
\end{proposition}

At issue is that the maximality of $\facet$ in $\buil(Y,H,X) \cap \apart$ 
does not imply maximality of $\facet^\buil$ in $\buil(Y,H,X)$; see 
Example~\ref{exampledist}. 

\begin{proof}[Proof of Proposition~\ref{P:mine}]
Let $\mathcal{O} \in \nilpk$ and let $(\facet,v)$ represent 
an element of $\gamma^{-1}(\mathcal{O})$.  Let $(\facet_0,v_0)$
be a representative of $\gamma_d^{-1}(\mathcal{O})$.  
Let $(Y,H,X)$ and $(Y_0,H_0,X_0)$ be Lie triples corresponding
to $(\facet,v)$ and $(\facet_0,v_0)$, respectively.  Since
they represent the same orbit, these Lie triples are conjugate
under $\ratG$, so there exists a $g\in\ratG$
such that 
$$
\buil(Y,H,X) = \buil(Ad(g)Y_0, Ad(g)H_0,Ad(g)X_0) = g\buil(Y_0,H_0,X_0).
$$
Consequently the maximal generalized $r$-facets in $\buil(Y,H,X)$
and $\buil(Y_0,H_0,X_0)$ have the same dimension,
which is equal to $\dim(\facet_0)$ by hypothesis.  It follows
that $\dim(\facet) \leq \dim(\facet_0)$.  If equality
holds then $Ad(g^{-1})\facet$ and $\facet_0$ are $r$-associate by
their maximality in $\buil(Y_0,H_0,X_0)$; thus $(\facet,v) \sim (\facet_0,v_0)$
since they represent also the same orbit. 

\end{proof}

\begin{proposition} \label{P:shiftr} 
Let $(Y,H,X)$ be a Lie triple in $\ratg$, and let $\lambda$ be
a one-parameter subgroup adapted to $(Y,H,X)$.  Then we have
$$
\buil_r(Y,H,X) = \buil_0(Y,H,X) + \frac{r}{2}\lambda.
$$
In particular the sum is well-defined.
\end{proposition}

\begin{proof}
This is the content of \cite[Remark 5.1.5]{DeBacker}.
Set $\ratM =  C_{\ratG^0}(\lambda)$ and let $\algM$ be the Levi subgroup of $\algG$, defined over $\ratk$, such that $\algM(\ratk) = \ratM$.  
By \cite[Corollary 4.4.2]{DeBacker},
 $\buil(\ratM) = \buil(\ratG)^{\lambda(\RR^\times)}$ and
by \cite[Corollary 4.5.9]{DeBacker}, for any $r \in \mathbb{R}$ and
$x \in \buil(\ratG)$, one has $\buil_r(Y,H,X) \subseteq \buil(\ratM)$.
Hence for any $x \in \buil_0(Y,H,X)$ there exists 
a $\ratk$-split torus $\algT \subset \algM$ so
that $x \in \apart(\algT)$ and $\lambda \in X_*(\algT)$, so the sum
is well-defined in $\apart(\algT)$.

To see the equality, note first that for any such $\apart$ and $x\in \apart$, $x \in \buil_r(Y,H,X)$
is equivalent to 
$$
X \in \sum_{\substack{\psi \colon \psi(x)\geq r\\ \dot{\psi}(\lambda)=2}}\ratg_{\psi}
\;\text{and}\; Y  \in \sum_{\substack{\psi \colon \psi(x)\geq r\\ \dot{\psi}(\lambda)=2}}\ratg_{-\psi}.
$$
Hence the result follows by noting that
$$
\{ \psi \colon \psi(x)\geq 0, \; \dot{\psi}(\lambda)=2\} = \{ \psi \colon \psi(x+\frac{r}{2}\lambda)\geq r,\; \dot{\psi}(\lambda)=2 \}.
$$

\end{proof}

\begin{corollary} \label{C:justone}
Let $(Y,H,X)$, $\lambda$ be as above.
If $\facet_r \subset \apart$ corresponds to a maximal generalized $r$-facet
of $\buil_r(Y,H,X)$, then there exists $x \in \facet_r$ and an $s$-facet $\facet_s'$ such that 
$x + \frac{s-r}{2}\lambda \in \facet_s'$ and  $\facet_s'$ corresponds to a maximal
generalized $s$-facet in $\buil_s(Y,H,X)$.
\end{corollary}

\begin{proof}
By Proposition~\ref{P:shiftr} we have that $\facet_r + \frac{s-r}{2}\lambda \subset \buil_s(Y,H,X) \cap \apart$ and by dimension we have that this subset must meet a maximal $s$-facet in an open nonempty set.  More precisely, given $x \in \facet_r$, maximality implies that the only $\psi \in \Psi$ for
which $\psi(x)=r$ are among those for which $\dot{\psi}(\lambda)=2$; in which case we have
$\psi(x+\frac{s-r}{2}\lambda) = s$.  Therefore $x+\frac{s-r}{2}\lambda$ will lie in a maximal $s$-facet if and only if
for all other $\psi$, $\psi(x) \neq s - \frac{s-r}{2}\dot{\psi}(\lambda)$.  In particular, the set of such $x$ is dense in $\facet_r$.  

\end{proof}

Corollary~\ref{C:justone} shows that given a Lie
triple $(Y,H,X)$,  
identifying a maximal $r$-facet in $\buil_r(Y,H,X)$ for just
one value of $r$ determines the $s$-associativity class of $s$-facets 
which are maximal in $\buil_s(Y,H,X)$, for any $s$.
Thus, in Sections~\ref{S:sl} and \ref{S:sp}
it suffices to establish our correspondence for values of $r \in \mathbb{R}\setminus \mathbb{Q}$,
for example.

The following example illustrates the difficulties one needs to
address to explicitly realize the correspondence.

\begin{example} \label{exampledist}
Consider $\algG = \mathrm{Sp}_4$ and choose $r \in \mathbb{R} \setminus \mathbb{Q}$.  Consider  the rational
nilpotent orbit represented by
$$
X_1 = \left[ \begin{smallmatrix}
0 & 0 &  0 & 1 \\ 0 & 0 & 1 & 0
\\ 0 & 0 & 0 & 0 \\ 0 & 0 & 0 & 0  \end{smallmatrix} \right].
$$
Then the orbit $\ratG \cdot X_1$ contains also a representative of
the form
$$
X_0 = \left[ \begin{smallmatrix}
0 & 0 &  1 & 0 \\ 0 & 0 & 0 & t
\\ 0 & 0 & 0 & 0 \\ 0 & 0 & 0 & 0  \end{smallmatrix} \right].
$$
where $-t \in \RR^{\times 2}$.
Complete these to Lie triples $(X_1, H, Y_1)$ and $(X_0,H,Y_0)$
with $H = \diag(1,1,-1,-1)$.
Then we have
\begin{align*}
\buil(Y_1,H,X_1) \cap \apart &= H_{e_1+e_2-r}\\
\buil(Y_0,H,X_0) \cap \apart &= H_{2e_1-r} \cap H_{2e_2-r}.
\end{align*}
Choose $\facet_1 \in \buil(Y_1,H,X_1) \cap \apart$ and $\facet_0 \in \buil(Y_0,H,X_0) \cap \apart$
with $\dim(\facet_i)=i$, so that each g$\facet_i$ is maximal in 
$\buil(Y_i,H,X_i) \cap \apart$, and set 
\begin{align*}
v_1 &= X_1+ \ratg_{\facet_1+} \in V_{\facet_1}\\ 
v_0 &= X_0 + \ratg_{\facet_0+} \in V_{\facet_0}.
\end{align*}
By Lemma~\ref{L:DB1} we conclude that 
$$
\mathcal{O} \doteq \mathcal{O}(\facet_1,v_1) = \Ad(\ratG)X_1 = \Ad(\ratG)X_0 = \mathcal{O}(\facet_0,v_0).
$$
It follows that $(\facet_0,v_0) \in \gamma^{-1}(\mathcal{O})$ and that
$\facet_0$
 is a maximal $r$-facet in $\buil(Y_0,H,X_0) \cap \apart$
but that it does NOT correspond to a maximal generalized $r$-facet in $\buil(Y_0,H,X_0)$.
Hence $(\facet_0,v_0) \notin I_r^d$.

On the other hand we may conclude that $(\facet_1,v_1) \in I_r^d$, as follows.
Since every $r$-associativity class necessarily
meets the apartment $\apart$, then if $(\facet_1,v_1)$ 
were not distinguished, then we would have
$$
2 \geq \dim(\buil(Y_1,H,X_1)) > \dim(\buil(Y_1,H,X_1) \cap \apart) = 1.
$$
Now if $\dim(\buil(Y_1,H,X_1)) = 2 = \dim(\apart)$ it would follow that there exists 
another representative $X' \in \mathcal{O}$ such that for a
corresponding Lie triple $(X', H', Y')$, we have 
$\buil(Y',H',X') \cap \apart \supset \facet'$ with $\facet'$
an $r$-alcove.  However, for $r$ non-integral $V_{\facet'} = \{0\}$; hence the unique corresponding 
orbit under the DeBacker correspondence (for any $r$)
is the trivial one.   This contradicts the choice of $X'$.  Hence
$\dim(\buil(Y_1,H,X_1))=1$ and $(\facet_1,v_1)$ is distinguished.
\end{example}

\begin{remark}
As we'll see in Sections~\ref{S:sl} and \ref{S:sp}, Lemma~\ref{L:DB1} implies that given a good choice of representative of a nilpotent orbit $\mathcal{O}$ it is relatively easy to identify elements of $\gamma^{-1}(\mathcal{O})$.  The
difficult step in identifying $\gamma_d^{-1}(\mathcal{O})$
 is the necessity of verifying that
one has indeed chosen the $r$-associate class of maximal dimension.
As Example~\ref{exampledist} illustrates, to identify maximal
generalized $r$-facets in $\buil(Y,H,X)$ (and hence distinguished pairs)
it does not suffice 
to consider maximal $r$-facets of 
$\buil(Y,H,X) \cap \apart$ for an given apartment
$\apart$.

Conversely, suppose we begin with a list of representatives of all the 
classes of pairs $(\facet,v)$ with $\facet \in \apart$.  To identify
the distinguished pairs, apply Proposition~\ref{P:mine} and 
work inductively downwards on the
dimensions of the facets:  for each pair $(\facet,v)$, use
 Lemma~\ref{L:DB1} to determine
$\mathcal{O} = \mathcal{O}(\facet,v)$; if $\mathcal{O}$ has
not been obtained from any previous pair, then the pair $(\facet,v)$
is distinguished.   

One difficulty in this latter approach seems to be to generate a list
of  
representatives
$I_r^d$.  The number of
classes of facets varies with $r$ and, as DeBacker notes, identifying
orbits of distinguished elements in $V_\facet$ is not easy in general;
choosing $r$ irrational is helpful here, as it increases the number of 
equivalence classes of facets while decreasing the number of orbits
associated to each class of facets.
\end{remark}

\section{The rational nilpotent orbits of $\SLn(\ratk)$} \label{S:sl}

\subsection{Partition-type parametrization of rational nilpotent orbits of $\SLn$}

Let $\ratk$ be as in Section~\ref{SS:lietriples}, $\algG = \SLn$ and 
$\lambda$ be a partition of $n$ with corresponding nilpotent $\algG$-orbit 
$\mathcal{O}_\lambda$.
One representative of $\mathcal{O}_\lambda$ is $J_\lambda$; this representative is clearly $\ratk$-rational.

Suppose that $X, X' \in \mathcal{O}_\lambda(\ratk)$ and let
$\phi$ and $\phi'$ be corresponding $\ratk$-rational Lie triples.
Then there exists $g \in \GLn(\ratk)$
which 
conjugates $\phi$ to $\phi'$; so $g$ preserves the
direct sum decomposition \eqref{E:decomp}, as well as each of the weight
spaces $V(j)$.   Write $g_j$ for the restriction
of $g$ to $L(j)$.
On each nonzero $V(j) \simeq L(j) \otimes W_j$, we have
$g(v \otimes u) = (g_j v) \otimes u$ for all $u \in W_j$, 
so $\det(g\vert_{V(j)}) = \det(g_j)^j$. 
Taking $g_j$ to be the identity
if $L(j) = \{0\}$, we have
$\det(g) = \prod_j (\det(g_j))^j$.
This product takes values only in $(\ratk^\times)^m$, where $m = \gcd(j \colon L(j) \neq \{0\}) = \gcd(\lambda)$, whence the following result.

\begin{proposition} \label{P:slorbits}
Let $\lambda$ be a partition of $n$ and set $m = \gcd(\lambda)$.  For any
$d \in \ratk^\times$, define
$D(d) = \diag(1,1,\ldots,1,d)$.
\begin{enumerate}
\item For each $d \in \ratk^\times$, the matrix $J_\lambda D(d)$ represents a $\ratk$-rational orbit in $\mathcal{O}_\lambda(\ratk)$, and conversely every orbit has a
representative of this form.
\item The $\SLn(\ratk)$-orbits represented by $J_\lambda D(d)$ and $J_{\lambda'}D(d')$
coincide if and only if $\lambda = \lambda'$ and $d \equiv d'$ in $\ratk^\times/(\ratk^\times)^m$.
\end{enumerate}
Thus there are exactly $\vert \ratk^\times/(\ratk^\times)^m\vert$  rational orbits of $\SLn(\ratk)$ in $\mathcal{O}_\lambda(\ratk)$.
\end{proposition}

\begin{proof}
Let $X \in \mathcal{O}_\lambda(\ratk)$ and set $X' = J_\lambda$.
Then there exists   $g \in \GLn(\ratk)$ such
that  $gXg^{-1} = J_\lambda$. Set $d = \det(g)$; then
$D(d^{-1})g \in \SLn(\ratk)$ and 
$$(D(d^{-1})g)\;X\;(D(d^{-1})g)^{-1} = J_\lambda D(d).$$  
Now let $X = J_\lambda D(d)$ and $X' = J_{\lambda'}D(d')$; these
are nilpotent elements in $\ratg$.  Choose corresponding Lie triples $\phi$ and 
$\phi'$ respectively.  That $\lambda = \lambda'$ follows from Proposition~\ref{P:partitions} and rest from the discussion preceding Proposition~\ref{P:slorbits}.
 \end{proof}

It follows immediately, for instance, that if the smallest part of
the partition $\lambda$ defining the algebraic orbit of $X$ is $1$,
then the algebraic orbit contains a unique rational
orbit.   
More generally, we note that $\gcd(\lambda)$ divides $n$ and all divisors of $n$
occur for some partition $\lambda$.

\subsection{Correspondence with the DeBacker parametrization} \label{S:slncorrespondence}

We now determine the $r$-associativity class to which
each nilpotent orbit is to be associated via the DeBacker correspondence.
Assume now additionally  that if $\charp = 0$ then the residual characteristic
$\resp$ is greater than $3(h-1)$.

Given a partition 
$\lambda = (\lambda_1, \lambda_2, \cdots, \lambda_t)$ of $n$ and 
$D = \diag(d_1,d_2,\cdots, d_n) \in \ratT$,
set 
$m = \gcd(\lambda)$ and define the set
$$
I_\lambda = \{1,2,3,\cdots, n\} \setminus \{ \lambda_1, \lambda_1+\lambda_2, \cdots, \sum_i\lambda_i = n\}.
$$ 
This set identifies the nonzero entries of the matrix of $X = J_\lambda D$:
the value $d_{i+1}$ is the $(i,i+1)$st entry, for each $i \in I_\lambda$;
all others are zero.  Recall that we denote the simple roots of $\ratG$ by $\alpha_i=e_i-e_{i+1}$.


\begin{theorem} \label{T:slcorrespondence}
Let $\lambda$ and $D$ be as above.
Define
\begin{equation} \label{flambda}
H_{\lambda,D} = \bigcap_{i \in I_\lambda} H_{\alpha_i + \val(d_{i+1}) - r}
\quad \subset \apart
\end{equation}
and let $\facet$ be any maximal $r$-facet in $H_{\lambda,D}$.  We have
$X = J_\lambda D \in \ratg_\facet$; set $v$ to be its image in
$V_\facet$.  Then $(\facet,v) \in I_r^d$ and $\mathcal{O}(\facet,v) = \Ad(\ratG)X$.
\end{theorem}

\begin{proof}
As noted previously, the zero orbit corresponds to the $r$-associate class of
any $r$-alcove, so we may assume from now on that $X \neq 0$ and so $I_\lambda \neq \emptyset$.
Set $N = \vert I_\lambda \vert$; then $N = \rank X$.

Since the set $\{ \alpha_i \colon i \in I_\lambda \} \subseteq \Phi$ 
is linearly independent, $H_{\lambda,D}$ is nonempty and has dimension $n-1-\vert I_\lambda \vert$.
It is a union of $r$-facets; moreover, all maximal $r$-facets $\facet$ contained 
in $H_{\lambda,D}$ satisfy  $A(\facet,\apart) =H_{\lambda,D}$, and so are strongly $r$-associate.

Let $\facet$ be such a facet and  $x \in \mathcal{F}$.  Then recall
$$
\ratg_{x,r} = \ratt_r \oplus \sum_{\psi \colon \psi(x) \geq r} \ratg_\psi.
$$
By construction, $X = J_\lambda D \in  \bigoplus_{i \in I_\lambda}  \ratg_{\alpha_i +\val(d_{i+1})}$,
and $\alpha_i(x) +\val(d_{i+1}) = r$, so $X \in \ratg_{x,r}$.  Let $v$ be the image of $X$
in $V_{x,r} = \ratg_{x,r}/\ratg_{x,r+}$.

A Lie triple corresponding to $X$ is 
$(Y,H,X)$ with $H = H_\lambda$ and $Y = D^{-1} Y_\lambda$.
By the hypotheses on $\resp$, the
valuations of the nonzero entries of $H_\lambda$ and $Y_\lambda$ are all zero.
Thus $H_\lambda \in \ratt_0 \subset \ratg_{x,0}$.
Moreover, since $\val(d^{-1})=-\val(d)$
we deduce
$$
Y = D^{-1}Y_\lambda \in  \bigoplus_{i \in I_\lambda} \ratg_{-\alpha_i-\val(d_{i+1})},
$$
which is in turn  a subset of $\ratg_{x,-r}$.  Let $(w,h,v) \in V_{x,-r} \times V_{x,0} \times V_{x,r}$
be the image of $(Y,H,X) \in \ratg_{x,-r} \times \ratg_{x,0} \times \ratg_{x,r}$; then
these two Lie triples correspond as in Lemma~\ref{L:DB1}, and $\mathcal{O}(\facet,v) = \Ad(\ratG)X$.

Suppose now $(\facet_0,v_0) \in I_r^n$ is another pair with $\facet_0 \subset \apart$ such that
$\mathcal{O}(\facet_0,v_0) = \Ad(\ratG)X$.  We claim that $\dim(\facet_0) \leq n-1-N$,
from which we may conclude that $(\facet,v) \in I_r^d$ by maximality of dimension.  By Corollary~\ref{C:justone}, it suffices to show this for just one choice of $r$; let $r \in \mathbb{R} \setminus \mathbb{Q}$.

Let $x \in \facet_0$.
Let $(w_0,h_0,v_0) \in V_{x,-r} \times V_{x,0} \times V_{x,r}$ be a Lie triple
corresponding to $v_0$.  
By \cite[Proposition 4.3.2]{DeBacker}, we have that \emph{any} lift of $v_0$ 
to an element $X_0$ in $\ratg_{x,r}$ (and lying in the $2$-weight space 
defined by an adapted one-parameter subgroup) 
extends to a lift $(Y_0,H_0,X_0)$ of the Lie triple $(w_0,h_0,v_0)$.
Furthermore, since $r \notin \Z$ we have $\ratt_r/\ratt_{r+} = \{0\}$.
Thus we may choose a lift $X_0$ such that
\begin{equation} \label{E:lift}
X_0 \in \bigoplus_{\psi \colon \psi(x) = r} \ratg_{\psi}.
\end{equation}
Since $\rank X_0 = \rank X = N$, there exist at least $N$ rows of the matrix
of $X_0$ with nonzero, non-diagonal entries.  Permuting the indices
as necessary, we may assume these nonzero entries correspond to affine
roots
$$
\psi_i = e_i - e_{j_i} + n_i, \quad 1 \leq i \leq N < n, \;1 \leq j_i \leq n, \; i \neq j_i
$$
for some $n_i \in \mathbb{Z}$.
By \eqref{E:lift} we have that $\psi_i(x)=r$ for each $i$,
for each $x \in \facet_0$.  Thus any $x \in \facet_0$ satisfies the 
inhomogeneous 
system of $N+1$ linear equations in $n$ unknowns
\begin{align} \label{E:sys2}
(e_i - e_{j_i})(x) &= r -n_i, \quad 1 \leq i \leq N,\\
\notag \sum_{i=1}^n e_i(x) &= 0.
\end{align}
To conclude the desired result, it suffices to show that this system 
is not overdetermined.

For suppose it was.  Then there would
exist $c_1, c_2, \ldots, c_N, c$, not all zero, such
that
\begin{equation} \label{E:sys}
\sum_{i=1}^N c_i (e_i - e_{j_i}) + c\sum_{i=1}^ne_i = 0.
\end{equation}
Form a directed graph $\Gamma$ with vertices equal to the set of indices
$\{1,2,\ldots, n\}$ and an edge from $i$ to $j_i$ for each $i \in \{1,2,\ldots, N\}$.  It is an exercise in linear algebra to see that the system \eqref{E:sys}
has rank $N+1-k$ where $k$ is the number of distinct closed cycles
of $\Gamma$, and the characteristic functions $f$ of the closed cycles
parametrize a basis for the space of solutions to \eqref{E:sys}
by setting $c_i = f(i)$.
Let $B$ denote the set of vertices in one cycle.  Then the corresponding
basic solution of \eqref{E:sys} is $\sum_{i \in B} (e_i - e_{j_i})(x)=0$.
Together with \eqref{E:sys2} this implies
$$
0 = \sum_{i \in B} (e_i - e_{j_i})(x) \equiv  \vert B \vert r \mod \mathbb{Z},
$$
which contradicts the choice of $r \notin \mathbb{Q}$.  Hence $k = 0$ and we  
conclude that \eqref{E:sys2} has a solution space of dimension 
$n-1-N$.  Thus $\dim(\facet_0) = n-1-N$, as required.
 \end{proof}

One immediate corollary of the proof is the following result.

\begin{corollary} \label{C:dimsl}
If $(\facet,v)  \in I_r^d$ satisfies $\mathcal{O}(\facet,v) \subset \mathcal{O}_\lambda(\ratk)$ then $$\dim \facet = \vert \lambda \vert - 1,$$
where $\vert \lambda \vert$ denotes the number of parts, counted with multiplicity, in the partition $\lambda$.
\end{corollary}

\begin{proof}
We saw in the proof of Theorem~\ref{T:slcorrespondence} that 
if $X \in \mathcal{O}(\facet,v)$, then $\dim(\facet) = n-1-\rank(X)$.
When $\mathcal{O}(\facet,v)$ is a rational orbit of $\mathcal{O}_\lambda$, 
we have $\rank X = \rank J_\lambda = n - \vert \lambda \vert$.
 \end{proof}

\begin{remark}
Recall from \cite{CMcG} that $\dim_K(\mathcal{O}_\lambda)=n^2-\vert \lambda \vert^2 - \sum_{i=2}^n\left(\sum_{j \geq i} m_j\right)^2$, where
$m_i$ denotes the multiplicity of $i$ in $\lambda$.  Thus there is
no direct relationship between $\dim(\mathcal{O}(\facet,v))$ and $\dim(\facet)$.  (As we shall see in Corollary~\ref{C:dim}, $\dim(\facet)$ is not even
an invariant of the algebraic orbit.)
\end{remark}

In general it is difficult to identify $r$-associate classes of facets in $\apart$.  However, for the special case that these facets are defined by intersections of hyperplanes corresponding to \emph{simple} roots, we have a complete answer, as follows.

\begin{corollary} \label{C:assoc}
Let $S \subseteq \{1,2,\ldots, n-1\}$ and for each $i \in S$ let $k_i \in \mathbb{Z}$.   Suppose $\facet$ is an $r$-facet of maximal dimension in 
$$
\bigcap_{i \in S} H_{\alpha_i + k_i - r} \subset \apart.
$$
Then $\facet$ is $r$-associate to any $r$-facet of maximal dimension
in 
$$
H_S = \begin{cases}
\bigcap_{i \in S} H_{\alpha_i - r} & \text{if $n-1\notin S$;}\\
\bigcap_{i \in S, i\neq n-1} H_{\alpha_i - r} \cap H_{\alpha_{n-1}+K-r} & \text{if $n-1 \in S$,}
\end{cases}
$$
where $K$ is defined as follows.  Let $\lambda$ be the unordered partition
of $n$ determined by the maximally consecutive subsets of $S$ and let $x_l = \sum_{j\leq l} \lambda_j$.  Then  $K$ is any integer in the equivalence class modulo
$\gcd(\lambda)$ of the sum 
$$
-\sum_{l=1}^{\vert \lambda \vert} \sum_{j=1}^{\lambda_l-1} j k_{x_l+j-1}.
$$
\end{corollary}

\begin{proof}
Given $\facet$ as above, 
we can construct a nilpotent $X$ and its
image in $V_\facet$ such that $(\facet,v) \in I_r^d$ and 
$\Ad(G)X = \mathcal{O}(\facet,v)$ using the methods of the proof of Theorem~\ref{T:slcorrespondence}.
Namely, if for each $i \in S$ we choose 
$X_i \in \ratg_{\alpha_i+k_i}\setminus \ratg_{\alpha_i+k_i+1}$
then $X = \sum_{i\in S} X_i$ is such an element.  We may write
$X = J_\lambda D$ for some 
diagonal matrix $D$ whose diagonal entries $d_i$ satisfy $\val(d_{i+1})=k_i$
for each $i \in S$.

Now let $\facet_0$ be a maximal $r$-facet in $H_S$.
By DeBacker's theorem, if we can show that there exists $X_0$ 
so that with $v_0 = X_0 +V_{\facet_0+}$, 
$(\facet_0,v_0)$ is another distinguished
pair such that $\mathcal{O}(\facet_0,v_0) = \mathcal{O}(\facet,v)$,
then we may deduce that $\facet$ and $\facet_0$ are $r$-associate.  

By Proposition~\ref{P:slorbits} if we construct a $g\in GL_n$ such that
$gXg^{-1}=J_\lambda$, then $X$ lies in the same $SL_n$ orbit as
$X_0=J_\lambda D(d)$, for any $d \equiv \det(g) \mod \gcd(\lambda)$.
It is clear that $X_0$ is the lift of an element of $V_{\facet_0}$,
and hence that we are done, if $K \equiv \val(d) \mod \gcd(\lambda)$.

Note that the unordered partition $\lambda$ is defined by first choosing
$1 = x_1 < x_2 < \cdots < x_t \leq n$ so that for each $x_l$,
either $x_l \notin S \cup\{n\}$, or $x_l \in S\cup\{n\}$ and $x_l-1 \notin S$.
Then we set $\lambda_l = x_{l+1}-x_l$ and $\lambda_t = n+1-x_t$;
$\vert \lambda \vert = t$ is the number of parts in $\lambda$.

For each $l$ with $1\leq l \leq t$ 
consider the matrix
$J_{\lambda_l}D_l$ with $$D_l = \diag(d_{x_l}, d_{x_l+1}, \cdots, d_{x_l+\lambda_l-1}).$$  Then the diagonal matrix 
$$
g_l = \diag(\prod_{j=1}^{\lambda_l-1}d_{x_l+j}^{-1}, \prod_{j=2}^{\lambda_l-1}d_{x_l+j}^{-1}, \cdots, d_{x_l+\lambda_l-1}^{-1},1)
$$
satisfies $g_l(J_{\lambda_l}D_l)g_l^{-1} = J_{\lambda_l}$ (as does $ag_l$
for any nonzero scalar $a$) and
$\det(g_l) = \prod_{j=1}^{\lambda_l-1} (d_{x_l+j}^{-1})^{j}$.  
Since
$X$ is the direct sum of the $J_{\lambda_l}D_l$, and $\val(d_j) = k_{j-1}$,
we have
$$
\val(\det(g)) = \sum_{l=1}^t \sum_{j=1}^{\lambda_l-1} -j k_{x_l+j-1},
$$
as we wished to show.
 \end{proof}

\section{The rational nilpotent orbits of $\Spn(\ratk)$} \label{S:sp}

\subsection{Nilpotent orbits} \label{SS:sppartition}
Let $\ratk$ be a local non-Archimedean field of characteristic zero 
or of characteristic $\charp>3(h-1)$  and assume additionally
that the residual characteristic $\resp$ is odd.

Let $\lambda$ be a partition of $2n$ in which odd parts occur with
even multiplicity.  Then by Proposition~\ref{P:partitions}, there is
an algebraic nilpotent adjoint orbit $\mathcal{O}_\lambda$ of $\algG = \Spn$,
the $k$-points of which form one
or more
rational orbits under the action of $\ratG = \Spn(\ratk)$.
These rational orbits are parametrized by 
isometry classes of quadratic forms, as described in Proposition~\ref{P:sp} below. 
 This theorem is derived using the
notation and approach of \cite[Chap 9.3]{CMcG}; the result
is functionally equivalent to \cite[I.6]{Waldspurger}.

\begin{proposition} \label{P:sp}
Let $\lambda$ be a partition of $2n$ and 
write $m_j$ for the multiplicity of $j$ in $\lambda$.
Suppose
$m_j$ is even whenever $j$ is odd.
The $\ratG$-orbits in $\mathcal{O}_\lambda(\ratk)$ are parametrized
by $n$-tuples 
$$
\overline{\Q} = (\Q_2, \Q_4, \cdots, \Q_{2n})
$$
where $\Q_j$ represents the isometry class 
of a nondegenerate quadratic form over $\ratk$ of dimension $m_j$
(taking $\Q_j = 0$ if $m_j=0$).
\end{proposition}

\begin{proof}
Let $X \in \mathcal{O}_\lambda(\ratk)$ and let $\phi \subset \ratg$ be a corresponding
Lie triple.  Then under $\phi$ the symplectic vector space $V$ decomposes as
$$
V = \bigoplus_{j \in \lambda} V(j) = \bigoplus_{j \in \lambda} L(j) \otimes W_j,
$$
with $\dim(L(j)) = m_j$, and each $V(j)$ a symplectic vector space.  The restriction of $\langle , \rangle$ to $V(j)$
naturally induces a nondegenerate form $(,)_j$ on the lowest weight space $L(j)$ via the formula
\begin{equation} \label{E:form}
(v,w)_j  = \langle v, X^{j-1}w \rangle \quad \forall v,w \in L(j).
\end{equation}
Note that this form is symplectic if $j$ is odd; such a form exists only if $\dim L(j) = m_j$ is
even, and then it is unique up to equivalence.
If $j$ is even, the form \eqref{E:form} is symmetric and nondegenerate; such forms are not unique (\emph{cf.} the first column of Table~\ref{T:numforms}).
Given $X' \in \mathcal{O}_{\lambda}(\ratk)$ and a corresponding Lie triple $\phi'$, 
suppose there exists a $g \in \ratG$ satisfying $\Ad(g)\phi = \phi'$.
Then the restriction $g_j$ of $g$ to each $V(j)$ induces
an isometry between $(L(j), (,)'_j)$ and $(L(j),(,)_j)$.
Conversely, any collection of such isometries
lifts to an element of $G = \Spn(\ratk)$.

Finally,  given any choice of nondegenerate symmetric form on
$L(j)$, for each even $j$, one can use \eqref{E:form} and
\eqref{E:decomp} to define  
a symplectic form on $V(j)$, and hence build a symplectic form $\langle,\rangle'$ on $V$,
with the property that $\phi \subset \mathfrak{sp}(V,\langle,\rangle')$. 
Since $\mathfrak{sp}(V,\langle,\rangle') \simeq \ratg$, it
follows that all equivalences classes of nondegenerate symmetric forms (equivalently, of
nondegenerate quadratic forms) on $L(j)$ arise for some choice of $X \in \mathcal{O}_\lambda(\ratk)$.
 \end{proof}

\begin{example}
Consider the groups $\mathrm{Sp}_4$ and $\mathrm{Sp}_6$, and suppose
$\charp = 0$ or $\charp > 3(h-1)$ so that Proposition~\ref{P:partitions} can 
be applied.
In Tables~\ref{T:examplesp4} and
\ref{T:examplesp6}, respectively,  we enumerate the 
algebraic orbits by partition $\lambda$,
and deduce the number of rational orbits in $\mathcal{O}_\lambda(\ratk)$
using Proposition~\ref{P:sp} and Table~\ref{T:numforms}.  The dimension
of each orbit can be  determined from the partition; see \cite[Chap 6]{CMcG}.
We include the dimension of the $r$-facet $\facet$ associated to $\mathcal{O}$
as a preview of Corollary~\ref{C:dim}.

\begin{table} 
\begin{tabular}{|ccccc|}
\hline Name & Partition & \# Rational Orbits & $\dim(\mathcal{O})$ & $\dim(\facet)$ \\ \hline
Principal & $[4]$ & $4$ & $8$ & $0$\\
Subregular & $[2^2]$ & $7$ & $6$ & $0$ or $1$  \\
Minimal & $[2,1^2]$ & $4$ & $4$ & $1$ \\
Zero & $[1^4]$ & $1$ & $0$ & $2$ \\ \hline
\end{tabular}
\caption{Rational nilpotent orbits of $\mathrm{Sp}_4(\ratk)$. \label{T:examplesp4}}
\end{table}

\begin{table}
\begin{tabular}{|ccccc|}
\hline Name & Partition & \#Rational Orbits &  $\dim(\mathcal{O})$ & $\dim(\facet)$ \\ \hline
Principal & $[6]$ & $4$ & $18$ & $0$\\
Subregular & $[4,2]$ & $16$ & $16$ & $0$ \\
 & $[4,1^2]$ & $4$ & $14$ & $1$ \\
 & $[3^2]$ & $1$ & $14$ & $1$ \\
 & $[2^3]$ & $8$ & $12$& $0$ or $1$ \\
 & $[2^2, 1^2]$ & $7$ & $10$ & $1$ or $2$ \\
Minimal & $[2,1^4]$ & $4$ & $6$ & $2$\\
Zero & $[1^6]$ & $1$ & $0$ & $3$ \\ \hline
\end{tabular}
\caption{Rational nilpotent orbits of $\mathrm{Sp}_6(\ratk)$. \label{T:examplesp6}}
\end{table}
\end{example}

To achieve our orbit correspondence, we need ``best'' representatives
of the different quadratic forms, which is the goal of the next
subsection.

\subsection{Classification of quadratic forms} \label{SS:lam}
In this subsection, $\ratk$ may be any local field of
characteristic $0$ or odd.  Let $\ep \in \ratk^\times$ be an non-square such that $\val(\ep)=0$.

Let $\Q$ be a quadratic form on an $m$-dimensional vector space over $\ratk$; 
write $Q$ for
a matrix representing $\Q$.
Define $\dim(\Q) = m$ and $\textrm{Det}(\Q)$ as the class of $\det(Q)$
in $\ratk^\times/(\ratk^\times)^2$.
For $a, b \in \ratk^\times$, the \emph{Hilbert symbol} $(a,b)_\ratk$ takes
values in $\pm 1$.  It equals $1$ if and only if $ax^2 + by^2 = 1$
has a solution $(x,y) \in \ratk^2$;  see Table~\ref{Hilbert}.
Given a diagonal representative $Q=\diag(a_1, a_2, \cdots, a_m)$ of 
$\Q$, the Hasse invariant of $\Q$ is defined as
$\Hasse(\Q) = \prod_{i<j} (a_i,a_j)_\ratk$.

\begin{table}
\begin{tabular}{|c|rrrr||c|rrrr|}
\hline
\multicolumn{5}{|c||}{Case: $-1 \in {\ratk^\times}^2$} & \multicolumn{5}{c|}{Case: $-1 \notin {\ratk^\times}^2$}\\ \hline
  $(a,b)_\ratk$    & $1$  & $\ep$ & $\p$ & $\ep \p$ & $(a,b)_\ratk$    & $1$ & $\ep$ & $\p$ & $\ep \p$\\ \hline
$1$ & $1$  & $1$   & $1$  & $1$      & $1$ & $1$ & $1$   & $1$  & $1$ \\
$\ep$ & $1$  & $1$   &$-1$  & $-1$  & $\ep$ & $1$  & $1$   &$-1$  & $-1$\\
$\p$ &  $1$   &$-1$  & $1$   &$-1$  & $\p$ &  $1$   &$-1$  &$-1$  & $1$\\
$\ep \p$ &  $1$   &$-1$  &$-1$  & $1$ & $\ep \p$ &  $1$   &$-1$  & $1$   &$-1$  \\ \hline
\end{tabular}
\caption{Explicit values of the Hilbert symbol on representatives of $\ratk^\times/(\ratk^\times)^2$ when $\resp \neq 2$.     \label{Hilbert}}
\end{table}

The following theorem is well-known;
see for example  \cite[Theorem VI.2.12]{Lam}.

\begin{theorem} \label{T:lam}
Let $\ratk$ be a local field of characteristic $0$ or odd.
Two nondegenerate quadratic forms $\Q$ and $\Q'$ over $\ratk$
are isometric if and only if 
$$
\dim(\Q) = \dim(\Q'), \quad \textrm{Det}(\Q) = \textrm{Det}(\Q'), \quad \text{and} \quad
\Hasse(\Q) = \Hasse(\Q').
$$
\end{theorem}

A quadratic form $\Q$ is called \emph{anisotropic} if there is no
nonzero $x$ such that $\Q(x)=0$, and \emph{isotropic} otherwise.
Following \cite[VI.2]{Lam}, we list the number of nondegenerate quadratic forms, together with
the number of those which are anisotropic, in Table~\ref{T:numforms}.
\begin{table}
\begin{tabular}{|c|c|c|}
\hline $\dim(\Q)$ & \# Quadratic Forms & \# Anisotropic Forms\\ \hline
1 & 4 & 4\\
2 & 7 & 6\\
3 & 8 & 4\\
4 & 8 & 1\\
$\geq$ 5 & 8 & 0 \\ \hline
\end{tabular}
\caption{Number of equivalence classes of nondegenerate quadratic forms over $\ratk$, $\resp\neq 2$, together with the number thereof which are anisotropic.} \label{T:numforms}
\end{table}
A key example of a nondegenerate isotropic quadratic form is the \emph{hyperbolic
plane}, which can be represented
by the matrix
\begin{equation} \label{Q:iso}
q_0 = \left[ \begin{matrix} 0 & 1 \\ 1 & 0 \end{matrix} \right].
\end{equation}
Thus $\dim(q_0)=2$, $\textrm{Det}(q_0) = -1$ and $\Hasse(q_0) = 1$.
Any orthogonal direct sum of hyperbolic planes is called a hyperbolic space.
Every nondegenerate quadratic form $\Q$ may be uniquely decomposed as a
direct sum of hyperbolic and anisotropic forms, which we may write as
\begin{equation} \label{sum}
\Q = q_0^m \oplus \Q_{\text{aniso}}
\end{equation}
where $q_0^m$ denotes the direct sum of $m$ copies of $q_0$, for some $m \leq \frac12 \dim(\Q)$.

\begin{lemma}  Suppose $\resp \neq 2$.
Given a quadratic form $\Q$, its anisotropic part 
$\Q_{\text{aniso}}$ is either the zero subspace, or one of the 15
anisotropic quadratic forms in Table~\ref{Table:anisotropic}.
\end{lemma}

\begin{table}
\begin{tabular}{|c|cclc|} 
\hline Dimension & Det($\Q$) & Hasse($\Q$) & Representative & \\  \hline
$1$ & $1$ & $1$ & $1$ &\\
$1$ & $\ep$ & $1$ & $\ep$ &\\
$1$ & $\p$ & $1$ & $\p$ &\\
$1$ & $\ep \p$ & $1$ & $\ep \p$ & \\ \hline
$2$ & $\alpha$ & $1$ & $\diag(1,\alpha)$ & \\
$2$ & $\alpha$ & $-1$ & $\diag(\p,\alpha \p)$ &\\
$2$ &  $tt'\p$ & $(t,t'\p)_k$ & $\diag(t,t'\p)$ & $t, t'\in \{1, \ep\}$\\ \hline
$3$ & $t$ & $-1$ & $\diag(\alpha t,\p,\alpha \p )$ & $t \in \{1,\ep\}$\\
$3$ &  $\alpha t \p$ & $(\alpha, \p)_k$ & $\diag(1, \alpha, t \p)$ & $t \in \{1, \ep\}$\\ \hline
$4$ &  $1$ & $-1$ & $\diag(1,-\ep,-\p,\ep \p)$ &\\ \hline
\end{tabular}
\caption{Explicit representatives of the 15 equivalence classes of anisotropic quadratic forms over $\ratk$, $\resp \neq 2$.  Here, we abbreviate  $\alpha = \ep$ if $-1 \in {\ratk^\times}^2$, and $\alpha = 1, \ep = -1$, otherwise. \label{Table:anisotropic}}
\end{table}

\begin{proof}
The values for dimensions $1$ and $4$ are from \cite[Theorem VI.2.2]{Lam}.
For the rest, one produces representatives of all equivalence classes
of quadratic forms from among diagonal matrices with entries in the set
$\{1,\ep,\p,\ep\p\}$, and then calculates their invariants.  
Note that any isotropic form of dimension $3$ must be of the form
$q_0 \oplus q_1$ for some $q_1 \in \{1,\ep,\p,\ep\p\}$, and so
are easy to eliminate from the list.
The
final result in Table~\ref{Table:anisotropic} is condensed by setting
$\alpha = \ep$ if $-1 \in {\ratk^\times}^2$, but $\alpha = 1$ and $\ep = -1$ 
if $-1 \notin {\ratk^\times}^2$.
 \end{proof}

Given a quadratic form $\Q$, a matrix representative of the form
$Q=q_0^m \oplus \Qani$, where $\Qani$ is one of the diagonal
matrices given in Table~\ref{Table:anisotropic}, will henceforth be called
a \emph{minimal matrix  representative} of $\Q$.

\subsection{Explicit parametrization of nilpotent orbits} \label{SS:sporbits}
We return to the hypotheses on $\ratk$ from Section~\ref{SS:sppartition}.
In \cite[Ch 5.2]{CMcG}, Collingwood and McGovern construct 
explicit Lie triples representing each of the algebraic
orbits $\mathcal{O}_\lambda$.  In this subsection, we construct
triples for each of the rational orbits in $\mathcal{O}_\lambda(\ratk)$
using similar ideas, though we must group the indices
in a slightly different way.

Let $\lambda$ be a partition of $2n$ such that each odd part occurs with
even multiplicity, and let $m_j$ denote the multiplicity of $j$ in $\lambda$.
Let $\overline{\Q} = (\Q_2, \Q_4, \cdots, \Q_{2n})$ be an $n$-tuple of 
quadratic forms corresponding to $\lambda$ and choose a minimal
matrix representative $Q_i$ for each $\Q_i$.   We construct a 
representative $X$ of the corresponding nilpotent orbit in $\ratg$
by first producing a decomposition of $\ratg$ of the form \eqref{E:decomp}
which corresponds to the partition $\lambda$, and then defining $X$ by
its action on each symplectic subspace $V(j)$.

Denote by $\{ p_1, p_2, \cdots, p_n, q_1, q_2, \cdots q_n \}$ an ordered
symplectic basis for $V$.  By this we mean that $\langle p_i, q_j \rangle = \delta_{ij}$,
$\langle q_i, p_j \rangle = -\delta_{ij}$ and
all other pairings are zero.
For each $i \in \{ 1,2,\ldots, 2n\}$, define a set of indices
$$
s_i = \sum_{j<i} \frac12 j m_j \in \mathbb{Z};
$$
so $0 = s_1 \leq s_2 \leq \ldots \leq s_{2n} \leq n$.
For each $j$ such that $m_j \neq 0$, 
construct the symplectic subspace $V(j)$ as
\begin{equation} \label{E:basis}
V(j) = \textrm{span}\{ p_{s_j + 1}, \ldots, p_{s_j+\frac12 jm_j}, 
q_{s_j+1}, \ldots, q_{s_j+\frac12 jm_j} \}.
\end{equation}
Then we have $V = \oplus_{j \colon m_j \neq 0} V(j)$, so to define $X$ it suffices to
give its action on each such nonzero $V(j)$.

If $j$ is odd, let $\mu$ be the partition of $\frac12 jm_j$
given by $\frac12 m_j$ copies of $j$, and 
define the restriction of $X$ to $V(j)$
with respect to the basis \eqref{E:basis} by
\begin{equation} \label{odd}
X|_{V(j)} = \left[ \begin{matrix}
J_{\mu} & 0 \\
0 & -J_\mu^\dagger
\end{matrix} \right].
\end{equation}
If $j = 2N$ is even, 
let 
$Z$ denote the $m_j(N-1) \times m_j(N-1)$ zero matrix.
Then define the restriction of $X$ to $V(j)$
with respect to the basis \eqref{E:basis} by
\begin{equation} \label{even}
X|_{V(j)} = \left[ \begin{matrix}
J_{Nm_j}^{m_j} & Z \oplus (-1)^{N} Q_{j} \\
0 & -(J_{Nm_j}^{m_j})^\dagger
\end{matrix} \right].
\end{equation}
Note that $J_{Nm_j}^{m_j} \doteq (J_{Nm_j})^{m_j} = J_{N} \otimes I_{m_j}$, where
$I_{m_j}$ the $m_j \times m_j$ identity matrix.

\begin{proposition} \label{P:representatives}
Let $\lambda$ be as above.  
The matrix $X\in \ratg$ defined by \eqref{odd} and \eqref{even} 
is a representative of the $\ratG$-orbit in $\mathcal{O}_\lambda(\ratk)$
corresponding to the $n$-tuple of quadratic forms $\overline{\Q}$.
\end{proposition}

\begin{proof}
We complete $X$ to a Lie triple $(Y,H,X)$ in $\ratg$, and then verify 
that $V$ decomposes according to $\lambda$ and that the
resulting form \eqref{E:form} coincides with $\Q_j$ for each
even $j$.  As done for $X$, we define $H$ and $Y$ by their action on 
the subspaces $V(j)$.

Following the notation of \eqref{HrYr}, for each odd $j$, let $H_{j}^{\oplus \frac12 m_j}$
denote the direct sum of $\frac12 m_j$ copies of the
diagonal
matrix $H_{j}$ and let
$H|_{V(j)} = H_{j}^{\oplus \frac12 m_j} \oplus -H_{j}^{\oplus \frac12 m_j}$.
Similarly, we may set
$Y|_{V(j)} = Y_{j}^{\oplus \frac12 m_j} \oplus -\left( Y_{j}^{\oplus \frac12 m_j}\right)^\dagger$.

For each even $j = 2N$, let $H_{Nm_j}$ be the $Nm_j \times Nm_j$ matrix given by $H_{Nm_j} = (j-1)I_{m_j} \oplus (j-3)I_{m_j} \oplus \cdots \oplus I_{m_j}$, and set
\begin{equation} \label{E:lowestweightspace}
H|_{V(j)}  = H_{Nm_j} \oplus -H_{Nm_j}.
\end{equation}
Similarly, set 
$$
Y_{Nm_j} = \left( J_N^\dagger \; \diag(j-1, 2(j-2), \ldots, (N-1)(N+1), 0) \right) \otimes I_{m_j}
$$
and let
$$
Y|_{V(j)} = \left[ \begin{matrix}
Y_{Nm_j} & 0 \\
Z \oplus (-1)^{N}N^2 Q_{j}^{-1} & -Y_{Nm_j}^\dagger
\end{matrix} \right].
$$
One checks directly that each $(Y|_{V(j)},H|_{V(j)},X|_{V(j)})$ is a
 Lie triple in the 
Lie subalgebra $\mathfrak{sp}(V(j))$ and hence
by orthogonality $(Y,H,X)$ is a  Lie triple in $\ratg$.
By construction the corresponding decomposition of $V$
into isotypic subspaces under this Lie triple corresponds
to the partition $\lambda$ so $X \in \mathcal{O}_\lambda(\ratk)$.

It remains to verify that when $j$ is even, the
quadratic form \eqref{E:form} is isometric to $\Q_{j}$.
First note that by \eqref{E:lowestweightspace}, with respect to the basis \eqref{E:basis} of $V(j)$,
the lowest weight space of this isotypic component is 
$$
L(j) = \textrm{span}\{ q_{s_j+1}, q_{s_j+2}, \ldots, q_{s_j+m_j} \}.
$$
Given $v,w \in L(j)$, write $\overline{v}, \overline{w}$ for
their coordinate vectors with respect to this basis.

By construction, we have that
$$
X^{j - 1}|_{V(j)}
= \left[ \begin{matrix} 
0 & -Q_j\oplus Z \\
0 & 0
\end{matrix} \right],
$$
which allows us to deduce directly that 
$\langle v,X^{j-1}w \rangle = \overline{v}^\dagger Q_j\overline{w}$, as required.
 \end{proof}

\subsection{Correspondence with the DeBacker parametrization} \label{SS:spdebacker}
Suppose now that $\resp > 3(h-1)$.
As in Section~\ref{S:slncorrespondence}, we assert 
that our choice of representative defines the equivalence class
of the corresponding pair in $I_r^d$.  To make this precise, we
need to identify the affine roots which occur in the expression of 
$X$ as a sum of root vectors; Theorem~\ref{T:debackerspn} below
says that the orbit of $X$ is associated to the corresponding intersection
of affine $r$-hyperplanes.  Unlike the case of $SL_n$,  however,
the roots which arise are not generally simple, so the precise
statement is more cumbersome.

Set our notation as in Section~\ref{SS:sporbits} and
suppose that $X \in \mathcal{O}_\lambda(\ratk)$ is as defined in
Proposition~\ref{P:representatives}.  For each odd $j$, 
let 
$$
I_{j} = \{1,2, \cdots, \frac12 jm_j\} \setminus \{ j, 2j, \cdots, \frac12jm_j\}
$$
and let $S_{j} = S_{j}^1$ denote the set of simple roots
$$
S_{j}^1=\{e_{s_j+k}-e_{s_j+k+1} \colon k \in I_{j} \}.
$$
For each even $j$, suppose $Q_j = q_0^m \oplus Q_{aniso}$, with $0 \leq 2m \leq m_j$ and set $M_j = (\frac12 j - 1)m_j$ for simplicity.  Then
we 
take instead 
$S_{j} = S_{j}^1 \cup S_{j}^2$, where these are the
sets of positive roots defined by
\begin{align*}
S_{j}^1=\{ & e_{s_j + k} - e_{s_j + k+m_j} \colon 1 \leq k \leq M_j \} \\
& \cup \{ e_{s_j + M_j + 2i-1} + e_{s_j + M_j + 2i} \colon i \in \{1, 2, \cdots, m\} \} 
\end{align*}
and 
$$
S_{j}^2 = \{ 2e_{s_j + M_j +i} \colon 2m < i  \leq m_j \}.
$$
Finally, if $Q_{aniso} = \diag(a_{2m+1}, a_{2m+2}, \cdots, a_{m_j})$, define
for $\alpha_i = 2e_{s_j + M_j +i}$ the integer $v_{\alpha_i} = \val(a_i)$ for each $i \in \{2m+1,2m+2,\cdots, m_j\}$.

Now let $S = \bigcup_{j \colon m_j \neq 0} S_{j}$
and define $H_{\lambda, \overline{\Q}}$ to be the common intersection
of all the hyperplanes $H_{\alpha-r}$ for $\alpha \in S_j^1$ and
$H_{\alpha+v_\alpha-r}$ for $\alpha = S_{2j}^2$.

\begin{theorem} \label{T:debackerspn}
The affine subspace $H_{\lambda,\overline{\Q}} \subset \apart$ is a nonempty union of $r$-facets.
Let $\facet$ be any maximal $r$-facet in $H_{\lambda,\overline{\Q}}$, and let
$v$ denote the projection of $X$ in $V_\facet$.  Then $(\facet,v) \in I_r^d$
and $\mathcal{O}(\facet,v) = \Ad(\ratG)X$.
\end{theorem}

\begin{proof}
The first assertion follows from the construction of $S$ as a linearly
independent subset of $\Phi$, and of $H_{\lambda, \overline{\Q}}$ as
an intersection of hyperplanes of the form $H_{\psi-r}$.
Given $\facet \subset H_{\lambda, \overline{\Q}}$ of maximal dimension,
and $x \in \facet$, we deduce from the proof of Proposition~\ref{P:representatives} that $X \in \ratg_{x,r}$, $H \in \ratg_{x,0}$ and $Y \in \ratg_{x,-r}$.
Furthermore, by the hypotheses on the residual characteristic, this
Lie triple projects onto a Lie triple 
$(w,h,v) \in V_{x,-r} \times V_{x,0} \times V_{x,r}$.
Thus we deduce that $\mathcal{O}(\facet,v) = \Ad(\ratG)X$,
as in the proof of Theorem~\ref{T:slcorrespondence}.  

It remains
to show that $(\facet,v)$ is distinguished, which we do by demonstrating
that $\dim(\facet) \geq \dim(\facet_0)$ for any other pair
$(\facet_0,v_0)$ such that $\mathcal{O}(\facet_0,v_0) = \Ad(\ratG)X$.

Start with $(\facet_0,v_0)$ and $x_0 \in \facet_0$.  Complete
$v_0$ to a Lie triple $(w_0,h_0,v_0)$ with adapted one-parameter
subgroup $\mu$ over $\resk$.  Using the argument in \cite[\S 4.3]{DeBacker}, we may 
conjugate the triple and $\mu$ by an element which fixes $x \in \apart$
to obtain a new Lie triple such that the adapted one-parameter subgroup
lies in $X_*(\algT)$.   Lift this to a Lie triple $(Y_0,H_0,X_0) \in 
\ratg_{x,-r} \times \ratg_{x,0} \times \ratg_{x,r}$.  

There exists a choice of $g$ normalizing $\ratT$ so that $\Ad(g)H_0$ is a dominant
toral element.  Since $g$ preserves $\apart$, we may without loss of
generality replace $\facet$ with $g\facet$ and $(Y_0,H_0,X_0)$ with
its $\Ad(g)$-conjugate.  (Note that in general this is not equal
to the Lie triple $(Y,H,X)$, since our $H$ is not generally 
dominant.)

We wish to show that if $\Ad(\ratG)X = \Ad(\ratG)X_0$, then
the dimension of $\facet_0$ is at most equal to $\dim(\facet)$.
As in the proof of Theorem~\ref{T:slcorrespondence}, we
assume that $r$ is irrational and begin by deducing that
\begin{equation*}
X_0 \in \bigoplus_{\psi \colon \psi(x)=r}\ratg_{\psi}.
\end{equation*}
Now let $\Phi(X_0) = \{\dot{\psi} \in \Phi \colon \psi(x)=r\}$; 
then we have 
\begin{equation}  \label{E:X0decomp}
X_0 \in \bigoplus_{\alpha \in \Phi(X_0)} \ratg_\alpha.
\end{equation}
Let $s = \dim(\textrm{span} \Phi(X_0))$.  Then $\dim(\facet_0) = n-s$ so 
our goal is to 
minimize $s$.    As was the case for $\SLn$, this is closely
related to the goal of minimizing the number of nonzero entries
in the matrix representing $X_0$; however, as some root spaces
contain matrices with two nonzero entries, there is more to say.

Let $V[i]$ denote the $i$-weight space of the toral element $H_0$.
Since $H_0$ is dominant, these weight spaces are strictly ordered
with respect to the basis \eqref{E:basis} of $V$.  
More precisely, we can define a decreasing list of indices $k_i$ by
$$
k_{i} = \sum_{j > i} \dim V[j], \quad  0 \leq i \leq 2n.
$$
Setting $k_{-i} = k_i$ and noting than $k_0 \leq n$, we have
\begin{equation} \label{E:basis2}
V[i] = \begin{cases}
\textrm{span}\{p_{k_i+1}, p_{k_i+2}, \cdots, p_{k_i+\dim V[i]} = p_{k_{i-1}}\} & \text{if $i>0$};\\
\textrm{span}\{q_{k_i+1}, q_{k_i+2}, \cdots, q_{k_i+\dim V[i]} = q_{k_{i-1}}\} & \text{if $i<0$;}\\
\textrm{span}\{p_{k_0+1}, \cdots, p_n, q_{k_0+1}, \cdots, q_n\} & \text{if $i=0$.}
\end{cases}
\end{equation}
In particular, $V[0]$, as well as $V[-i] \oplus V[i]$ for each $i \geq 1$,  are
symplectic subspaces of 
$V$.

For each $i$, the restriction of $X_0$ to the $i$-weight space gives a
map $X_0 \colon V[i] \to V[i+2]$; when $i \geq -1$, this map is surjective.
Let us decompose $X_0$ into a sum of simpler elements in $\ratg$ via 
these restricted maps.  

Since $V[-1]\oplus V[1]$ is a symplectic subspace of $V$,
the map $X_0^{(-1)}$ defined by extension by zero of the restriction
$X_0 \colon V[-1] \to V[1]$ is a well-defined  element of $\ratg$.  Similarly, for $i\geq 0$,
define $X_0^{(i)}$ to be the restriction of $X_0$ to the
domain  $V[i] \oplus V[-i-2]$, and zero on all other weight spaces; then $X_0^{(i)} \in \ratg$.
Altogether, we have a decomposition of $X_0$ as
$$
X_0 = \sum_{i \geq -1} X_0^{(i)}.
$$
By construction, these components are supported on disjoint subsets of
the basis \eqref{E:basis} of $V$ associated to the root system
$\Phi$.  It follows that we may decompose the set of roots $\Phi(X_0)$
occurring in \eqref{E:X0decomp} as a disjoint union
$$
\Phi(X_0) = \Phi_{(-1)} \cup \Phi_{(0)} \cup \cdots \cup \Phi_{(2n-2)}
$$
such that for each $i \geq -1$,  
\begin{equation} \label{E:X0basis2}
X_0^{(i)} \in \bigoplus_{\alpha \in \Phi_{(i)}}\ratg_\alpha.
\end{equation}
We can say more about these sets $\Phi_{(i)}$.  

Suppose first that $i \geq 1$. Then we deduce from \eqref{E:basis2} that
\begin{equation} \label{E:set}
\Phi_{(i)} \subseteq  \{ e_l-e_j \colon l \in \{ k_{i+2}+1, \cdots, k_{i+1}\}, j \in  \{k_i+1, \cdots, k_{i-1}\} \}.
\end{equation}
Furthermore, since the restriction $X_0 \colon V[i] \to V[i+2]$ is surjective,
the matrix $A$ representing this map has at least one nonzero entry in each
row.  In terms of the roots which occur in $\Phi_{(i)}$, one may deduce
that for each $l \in \{ k_{i+2}+1, \cdots, k_{i+1}\}$ there exists a $j_l \in \{k_i+1, \cdots, k_{i-1}\}$
such that $e_l - e_{j_l} \in \Phi_{(i)}$.  Hence we conclude that
$\vert \Phi_{(i)} \vert \geq \dim(\textrm{span} \Phi_{(i)}) \geq \dim V[i+2]$.

Next consider $i=0$.  Since $V[0]$ meets both $\textrm{span}\{p_1,\ldots, p_n\}$
and $\textrm{span}\{q_1,\ldots,q_n\}$, it follows that additionally $\Phi_{(0)}$ may contain  roots of the form
$e_l+e_j$, for indices $l$ and $j$ corresponding
to $V[0]$ and $V[2]$, respectively.
Consequently,
$$
\Phi_{(0)} \subseteq \{ e_l \pm e_j \colon l \in \{ k_{2}+1, \cdots, k_{1}\}, j \in  \{k_0+1, \cdots, n \} \}.
$$
Since $X_0^{(0)}$ has full rank, we deduce as above that 
$\vert \Phi_{(0)} \vert \geq \dim(\textrm{span} \Phi_{(0)}) \geq \dim V[2]$. 

Finally, suppose that $i=-1$.  The map $X_0^{(-1)}$ 
sends the span of $\{ q_{k_1+1}, \ldots, q_{k_0}\}$ onto
the span of $\{ p_{k_1+1}, \ldots, p_{k_0}\}$, so we have that
\begin{equation} \label{E:badcase}
\Phi_{(-1)} \subseteq \{ e_l + e_j \colon l,j \in  \{k_1+1, \cdots, k_0 \}\}.
\end{equation}
This time, however, there is little connection between the number of roots in $\Phi_{(-1)}$
and the rank $\dim V[1]$ of $X_0^{(-1)}$.  Namely, when $l=j$, a root
 vector corresponding to the root $e_l+e_j$ has a single nonzero entry,
whereas when $l \neq j$, it has two nonzero entries, in two distinct rows. 

To better understand the constraints on $\Phi_{(-1)}$ we consider the symmetric matrix $B$ 
representing $X_0 \colon V[-1] \to V[1]$
with respect to the bases \eqref{E:basis2}.  We have the
following Lemma.

\begin{lemma} \label{L:B}
If $X_0$ represents the $\ratG$-orbit defined by the pair $(\lambda, \overline{\Q})$ 
then the matrix $B$ represents the quadratic form
\begin{equation} \label{E:q}
-\Q_2 \oplus \Q_4 \oplus -\Q_6 \oplus \cdots \oplus (-1)^n \Q_{2n}.
\end{equation}
\end{lemma}

\begin{proof}
Given the Lie triple $(Y_0,H_0,X_0)$, consider the 
decomposition \eqref{E:decomp} of $V$ it determines.
For each $t \geq 1$, the lowest weight space $L(2t)$ of the isotypic component
$V(2t)$ (if nonzero) carries the quadratic form $\Q_{2t}$ of dimension $m_{2t} = \dim L(2t)$.
This form is given by $(v,w)_{2t} = \langle v, X_0^{2t-1} w \rangle$.
Note that $L(2t) \subseteq V[-2t+1]$. There is a natural map 
from $L(2t)$ to $V[-1]$ given by $v \mapsto X_0^t v$; denote
its image $E(2t)$.  Then $V[-1]$ decomposes as direct sum
\begin{equation} \label{E:splitv1}
V[-1] = \bigoplus_{t \geq 1} E(2t).
\end{equation}
The form $(,)_{2t}$ on $L(2t)$ induces
a form $(,)$ on $E(2t)$, for each $t$, and hence by orthogonal direct sum
a quadratic form on $V[-1]$.
It is given on each $E(2t)$ by the formula
$$
(X_0^t v, X_0^t w) = \langle X_0^t v, X_0X_0^t w \rangle = (-1)^{t} \langle v, X_0^{2t-1}w \rangle = (-1)^t (v,w)_{2t}.
$$
Since $V[-1]$ and $V[1]$ have complementary bases under $\langle,\rangle$
the Lemma follows.
 \end{proof}

Given all the above, we wish to bound the value of $\vert \Phi(X_0) \vert$.  We consider 
two possible cases.

In the first case, suppose that each subspace $E(2t)$ of the decomposition \eqref{E:splitv1} 
is spanned by a subset of $\{q_{k_1+1}, \ldots, q_{k_0}\}$.  Then
by orthogonality $X_0E(2t)$ has as basis the complementary subset of
$\{p_{k_1+1},\ldots, p_{k_0}\}$, and so the matrix of $B$ is block-diagonal,
corresponding to the decomposition \eqref{E:q}.  Hence the problem of minimizing 
$\vert \Phi_{(-1)} \vert$
is reduced to minimizing the roots required to represent each of the
quadratic forms $\Q_{2t}$, which by Section~\ref{SS:lam} implies
one should choose for each $\Q_{2t}$ the minimal matrix representative $q_0^m \oplus Q_{aniso}$ (up to order of the summands).

Consequently, in this first case, we deduce that the minimal
possible value of $\vert \Phi(X_0) \vert$ is exactly $\vert \Phi(X) \vert$,
and thus that $\dim(\facet_0) \leq \dim(\facet)$, as required.

In the second case, suppose that some of the subspaces $E(2t)$ are not
aligned with the symplectic basis.  The only way to reduce
the size of $\Phi_{(-1)}$ would be to choose a matrix representative 
$B$ of the quadratic form \eqref{E:q} which implies fewer roots 
than our preferred choice, above.  This in turn can only happen 
if we form a hyperbolic plane
from vectors coming from different isotypic components.  We claim
that any such reduction in
$\vert \Phi_{(-1)} \vert$ would be offset by a corresponding increase in some
$\vert \Phi_{(i)} \vert$ with $i \geq 1$, as follows.

Suppose that $B$ contains a summand $q_0$ at indices $t_1$ and $t_2$.
So we have 
$q_{t_1}, q_{t_2} \in V[-1]$ such that 
$$
X_0q_{t_1} = p_{t_2}\quad \text{and} \quad X_0q_{t_2} = p_{t_1}.
$$
If $q_{t_1}$ and $q_{t_2}$ lie in the same isotypic component of
$V$, then this hyperbolic plane falls under the analysis of
the first case so 
assume this is not the case.

For each $i=1,2$, consider the sequences
$(p_{t_i}, X_0p_{t_i}, X_0^2p_{t_i}, \cdots)$.  If there exists
an $i \in \{1,2\}$ and a least $w > 0$ such that $X_0^wp_{t_i}$ is
not again a scalar multiple of some $p_j$, then write $X_0^{w-1}p_{t_i} = cp_j$
and $X_0^wp_{t_i} = X_0(cp_j) = \sum_{i=1}^n c_ip_i$ with at least
two coefficients $c_k, c_l$ different from zero.   Now $p_j \in V[2s+1]$
for some $s \geq 1$, so by \eqref{E:X0basis2}, 
 it follows that $e_{k}-e_j, e_{l}-e_j \in \Phi_{(2s+1)}$.
Using rank arguments as before (this time counting columns with nonzero
entries), we conclude that $\vert \Phi_{(2s+1)} \vert \geq 1+ \dim V[2s+3]$.
 Thus the decrease by one in the size of $\Phi_{(-1)}$ is offset by an increase of
at least one in the size of $\Phi_{(2s+1)}$. 

So we may assume from now on that both sequences consist of
multiples of vectors from our symplectic basis.  That is,
for each $i$ we have indices $s(i,l)$ and an integer
$w_i$ so that up to nonzero scalar multiples
$$
(p_{t_i}, X_0p_{t_i}, X_0^2p_{t_i}, \cdots) = (p_{t_i}, p_{s(i,1)}, \cdots, p_{s(i,w_i)}, 0, \cdots).
$$
 Now note that for any $k \in \{t_i, s(i,1), \cdots, s(i,w_i) : i \in \{1,2\}\}$, if 
$X_0p_k = cp_j \neq 0$
then $\langle X_0q_j, p_k \rangle = -\langle q_j,X_0p_k\rangle = -\langle q_j,cp_j\rangle =c$, so $X_0q_j = cq_k + \text{other terms}$.
If the other terms were nonzero, then the same analysis as
above would yield an $s \geq 1$ such that $\vert \Phi_{(2s+1)} \vert$
is not minimal.  Hence we may assume $X_0q_j = cq_k$.

From the sequence above we have that $p_{s(i,w_i)}$ is a vector of 
highest weight
$2w_i+1$; hence $q_{s(i,w_i)}$ has weight $-2w_i-1$ since they are 
complementary in the symplectic basis.   By the preceding paragraph,
however, we have that the smallest invariant subspace containing 
$q_{s(i,w_i)}$ contains $p_{t_j}$,for $j\neq i \in \{1,2\}$, as
a highest weight vector; so necessarily $w_j \geq w_i$.
We deduce that 
$w_1=w_2=w$, which implies
$q_{t_1}$ and $q_{t_2}$ both lie in $E(2w+2)$, contrary to assumption.

Finally, we remark that given any number of distinct hyperbolic pairs
arising from $V[-1]$, the roots added to $\Phi(X_0)$ through the
above argument will be distinct.
We conclude that one cannot do better than $\Phi(X)$, so 
$\dim(\facet) \geq \dim(\facet_0)$.
 \end{proof}

\begin{corollary} \label{C:dim}
Let $(\lambda,\overline{\Q})$ represent a rational nilpotent orbit $\mathcal{O}$.
Suppose that $(\facet,v) \in I_r^d$ such that $\mathcal{O} = \mathcal{O}(\facet,v)$.  
For each $i$, let $a_i$ denote the dimension of the largest anisotropic
subspace of $\Q_{2i}$.  Then
$$
\dim(\facet) = \frac12\left(\vert \lambda \vert  - \sum_{i=1}^n a_i \right).
$$
\end{corollary}

\begin{proof}
Recall that $\vert \lambda \vert$ denotes the number of parts in the partition
$\lambda$, counting with multiplicity.

We have that $\dim \facet = \dim(H_{\lambda,\overline{\Q}}) = n - \vert S \vert$.
When $j$ is odd, $\vert S_{j} \vert = \vert I_{j} \vert = \frac12 jm_j - \frac12 m_j = \frac12 m_j (j-1)$.
When $j=2i$ is even, we have $2m+a_i = m_j$ so 
$\vert S_{j}^1 \vert = m_j(\frac12 j -1) + \frac12(m_j-a_i)$ and 
$\vert S_{j}^2 \vert = a_{i}$, so $\vert S_{j} \vert = \frac12 m_j(j-1)+\frac12a_{i}$.
Hence since $\sum jm_j = 2n$, we have $\vert S \vert = \sum_{j=1}^{2n} \frac12m_j(j - 1) + \frac12 \sum_{i=1}^n a_i
= \frac12 (2n - \sum_{j=1}^{2n} m_j + \sum_{i=1}^n a_i)$, and the result follows.
 \end{proof}

See Tables~\ref{T:examplesp4} and \ref{T:examplesp6} for examples.

\begin{remark}
Recall the notion of distinguished orbits, key to the Bala-Carter classification.  By \cite[Lemma 4.2]{Jantzen}, these may be characterized for $\Spn$ as the
orbits with partitions
consisting of distinct even parts.  So if $\mathcal{O}$ is distinguished, $a_i = 1$ for each $i$ and it follows from Corollary~\ref{C:dim}
that $\dim(\facet)=0$; the distinguished orbits correspond to vertices.
(This also holds for $\SLn$, where only the principal orbit
(corresponding to the partition $(n)$) is distinguished.) 
As we can see already from $Sp_4$ and $Sp_6$,
the converse is false in general.

Note also that the $r$-associativity classes of \emph{vertices} are easy to determine for any $r$;
there is one representative of each in the fundamental domain.
\end{remark}

\section{Some remarks on the case of small residual characteristic} \label{S:smallreschar}

The parametrization of nilpotent orbits via conjugacy classes
of Lie triples fails to hold in
small characteristic.   
Since this parametrization, over finite fields,
forms the backbone of DeBacker's results in \cite{DeBacker}, it is
not clear to what extent the DeBacker parametrization will
hold over $p$-adic fields with small residual characteristic.  
We explore this question in this section.

So let $\ratk$ be a $p$-adic field with residual characteristic $\resp$.
Then the partition-type parametrization of
nilpotent orbits discussed
in Section~\ref{S:sl} is valid for all $\resp$ and that in Section~\ref{S:sp}
is valid for all $\resp>2$.  In Section~\ref{S:debacker} we 
have assumed that $\resp >3(h-1)$.
Let us consider now values of $\resp$ less than this bound.

Consider the following essential characteristics
of the
correspondence $(\facet,v) \mapsto \mathcal{O}$:
\begin{enumerate}
\item That $\mathcal{O}$ is the unique nilpotent orbit of minimal
dimension meeting the coset $v = X+\ratg_{\facet+}$.
\item That $(\facet,v)$ corresponds to a maximal generalized $r$-facet
in $\buil(Y,H,X)$.
\item That $\gamma_d$ is bijective.
\end{enumerate}

We illustrate how (2) can fail when $\resp$ is small with
an example.

\begin{example} \label{examplep3}
Consider the group $\algG = \SLn$.
Suppose $n = 4$ and $\resp = 3$; let $\lambda = (4)$ be the full partition, corresponding to the principal nilpotent orbit.  
The rational orbits contained in $\mathcal{O}_\lambda(\ratk)$ 
are parametrized by
$$
X_d = J_\lambda D(d) = \left[ \begin{matrix} 0 & 1 & 0 & 0 \\ 0 & 0 & 1 & 0\\ 0 & 0 & 0 & d
\\ 0 & 0 & 0 & 0 \end{matrix} \right],
$$
with $d$ running over a list of representatives of the
distinct quartic classes in $\ratk^\times$.  
One has $H = H_\lambda = \diag(3,1,-1,-3)$ and 
$$
Y = D(d^{-1})Y_\lambda = \left[ \begin{matrix} 0 & 0 & 0 & 0 \\ 3 & 0 & 0 & 0\\ 0 & 4 & 0 & 0
\\ 0 & 0 & 3d^{-1} & 0 \end{matrix} \right].
$$
Let us suppose $\val(d) = 0$ for definiteness.   

Define $H_{\lambda,D(d)}$ as in Theorem~\ref{T:slcorrespondence}.
Let $\facet$ be a maximal $r$-facet in $H_{\lambda,D(d)}$; this is a
single vertex.  Then one readily verifies that the subgroups $\ratg_\facet$
and $\ratg_{\facet+}$ are of the following form (with the usual abuse
of notation):
$$
\ratg_{\facet} = \left[\begin{matrix} \PP & \RR & \RR & \RR \\
\PP & \PP & \RR & \RR \\
\PP & \PP & \PP & \RR \\
\PP & \PP & \PP & \PP 
\end{matrix} \right] 
\text{\;and\;} \ratg_{\facet+} = 
\left[\begin{matrix} 
\PP & \PP & \RR & \RR \\
\PP & \PP & \PP & \RR \\
\PP & \PP & \PP & \PP \\
\PP & \PP & \PP & \PP 
\end{matrix} \right].
$$
It follows  that any element in the coset $X_d + \ratg_{\facet+}$
has rank at least $3$, and so any nilpotent orbit meeting this coset
is regular, and hence of dimension at least (in fact equal to) 
$\dim(\Ad(\ratG)X_d)$.  We note directly that for all $g$ in the
upper triangular Borel subgroup $\ratB$, 
$\Ad(g)X_d \notin X_{d'} +\ratg_{\facet+}$ if $d \neq d' \mod (\RR^\times)^4$,
and that if $g$ is in any other Bruhat cell, then $\Ad(g)X_d$ does
not meet any such cosets.  Hence $\Ad(G)X_d$ is the unique nilpotent
orbit of minimum dimension meeting the coset $X_d+\ratg_{\facet+}$,
proving characterization (1). 

Now set $X = X_d$ for simplicity and let us determine $\buil(Y,H,X) \cap \apart$.  From our explicit
matrix choices for $Y$ and $X$ we have that
$Y \in \ratg_{x,-r}$ for some $x \in \apart$ if and only if
\begin{align*}
-\alpha_1(x) + \val(3) &\geq -r\\
-\alpha_2(x) + \val(4) & \geq  -r \\
- \alpha_3(x) + \val(3) & \geq  -r,
\end{align*}
whereas $X \in \ratg_{x,r}$ if and only if $\alpha_i(x) \geq r$ for 
$i \in \{1,2,3\}$.

Note that $\val(3) \geq 1$ and $\val(4) = 0$ in 
this case.  
 Thus $\buil(Y,H,X) \cap \apart = \{ x \in \apart \colon Y \in \ratg_{x,-r}, X \in \ratg_{x,r} \}$ is the set of all points $x \in \apart$ such that
\begin{align*}
r \leq &\;\alpha_1(x) \leq r + \val(3)\\
r \leq &\;\alpha_2(x) \leq r \\
r  \leq &\;\alpha_3(x) \leq r+\val(3).
\end{align*}
This region contains some 2-dimensional $r$-facets,
and thus $\facet$ is not maximal, so the characterization (2) fails
for this choice of $(\facet,v)$.

Let us now further show that a choice of $(\facet,v)$ satisfying (2) cannot
satisfy (1).

To fix our ideas, let us suppose $r \in (0,\frac14)$
and let $\facet'$ be the $r$-facet containing the point $x\in \apart$ such
that $\alpha_1(x)=\alpha_3(x) = \frac14$ and $\alpha_2(x) = r$.
Then $\facet'$ is a maximal $r$-facet in $\buil(Y,H,X) \cap \apart$.
Furthermore, for any $x \in \facet'$,
we have that $r \leq \phi(x) < 1+r$ for all positive roots $\phi \in \Phi^+$ and $-1+r < \phi(x) < r$
for all negative roots $\phi$, with exactly one equality (for $\phi = \alpha_2$).  
So the elements of $\ratg_{x,r}$
have their strictly upper triangular entries in $\RR$,
and all lower triangular entries in $\PP$.  
The strictness of all but one inequality
above implies that the elements of 
$V_{x,r} = \ratg_{x,r}/\ratg_{x,r+}$ may be represented by
matrices of the form
$$
v_a = \left[ \begin{matrix}
0 & 0 & 0 & 0 \\
0 & 0 & a & 0 \\
0 & 0 & 0 & 0\\
0 & 0 & 0 & 0
\end{matrix} \right]
$$
with $a \in \resk^\times$.  Similarly, we deduce that $\ratg_{x,-r}$
and $\ratg_{x,0}$ have the form
$$
\ratg_{x,-r} = \left[\begin{matrix} \RR & \RR & \RR & \RR \\
\PP & \RR & \RR & \RR \\
\PP & \RR & \RR & \RR \\
\PP & \PP & \PP & \RR 
\end{matrix} \right] , \quad 
\ratg_{x,0} = \left[\begin{matrix} \RR & \RR & \RR & \RR \\
\PP & \RR & \RR & \RR \\
\PP & \PP & \RR & \RR \\
\PP & \PP & \PP & \RR 
\end{matrix} \right].
$$
It is true that $(Y,H,X) \in \ratg_{x,-r} \times \ratg_{x,0} \times \ratg_{x,r}$.  However, now consider
instead the Lie triple
$$
Y' = \left[ \begin{matrix}
0 & 0 & 0 & 0 \\
0 & 0 & 0 & 0 \\
0 & 1 & 0 & 0\\
0 & 0 & 0 & 0
\end{matrix} \right], \quad
H' = \left[ \begin{matrix}
0 & 0 & 0 & 0 \\
0 & 1 & 0 & 0 \\
0 & 0 & -1 & 0\\
0 & 0 & 0 & 0
\end{matrix} \right], \quad
X' = \left[ \begin{matrix}
0 & 0 & 0 & 0 \\
0 & 0 & 1 & 0 \\
0 & 0 & 0 & 0\\
0 & 0 & 0 & 0
\end{matrix} \right].
$$
This triple certainly satisfies the condition that $(Y',H',X') \in \ratg_{x,-r} \times \ratg_{x,0} \times \ratg_{x,r}$, but now $\dim(\Ad(\ratG)X') < \dim(\Ad(\ratG)X)$.  Thus
$\Ad(\ratG)X$ is not the orbit of minimal dimension meeting the 
coset $X+\ratg_{x,r+}$.
Hence the distinguished pair $(\facet',v')$ (condition (2)) does not satisfy the minimality requirement (condition (1)).
\end{example}

This example illustrates that  
the set $\buil(Y,H,X)$ is too large to 
uniquely identify a facet 
which is associated to the nilpotent orbit $\Ad(\ratG)X$.

Nonetheless, through  Theorem~\ref{T:slcorrespondence}
and Theorem~\ref{T:debackerspn} there exist obvious candidates for
the ``distinguished pairs'' $(\facet,v)$, namely, those
 satisfying Proposition~\ref{P:mine}.  Unfortunately, Proposition~\ref{P:mine}
is a less satisfying definition because it is not intrinsic.  In any
case what still 
needs to be
proven is that property (3) continues to hold for these in small
residual characteristic.

For our final remark, we note that even this approach will encounter some obstacles for
the case of $\algG = \SLn$, as the Example~\ref{Example:notverygood}, below, illustrates.  We first recall a well-known result; one can prove it by adapting
the proof of \cite[Thm VI.2.22]{Lam}, which is the special
case $m=2$.

\begin{proposition} \label{P:cosets}
Let $m \geq 2$ and let $\mu_m(\ratk)$ denote the group of $m$th roots of unity in $\ratk$.  Then the number of distinct cosets of $(\ratk^\times)^m$ in $\ratk^\times$ is
$$
\vert \ratk^\times/(\ratk^\times)^m \vert = m \vert \RR^\times / (\RR^\times)^m \vert 
$$
and $\vert \RR^\times / (\RR^\times)^m \vert = \vert \mu_m(\ratk) \vert q^{\val(m)}$.
\end{proposition}

Applying now  Proposition~\ref{P:slorbits}, we have
a more precise result about the number of 
rational nilpotent orbits of
$\SLn$.
We may now present an example to illustrate how (3) may fail
when $\resp$ is not very good for $\algG$.

\begin{example} \label{Example:notverygood}
Suppose $\algG = \SLn$ and that $\resp$ is a prime which is not very good
for $\algG$, meaning that $\resp$ divides $n$.  
Let $\lambda$ be a partition of $n$ such that $(\gcd(\lambda),\resp) =m>1$.
Then  $\val(m) \geq 1$ and so by Proposition~\ref{P:cosets}, 
there are $q^{\val(m)}\vert \mu_m(k) \vert$ distinct  representatives
of $\ratk^\times/(\ratk^\times)^m$ of any given valuation.  In
particular, there exist distinct cosets represented by elements $d,d'$ of
the same valuation, such that their difference $d-d'$ has strictly
larger valuation.  Thus the  corresponding representatives of the nilpotent orbit
will descend to the same element in $\ratg_{\facet}/\ratg_{\facet+}$.
\end{example}

In this example, the number of rational orbits in $\mathcal{O}_\lambda(\ratk)$
was too large to be parametrized by $\ratG_x$ orbits of $v$
in $V_\facet$, for any collection of pairs $(\facet,v)$.  Thus the DeBacker
correspondence cannot hold.
We note that this issue is
also a problem 
for the Bala-Carter classification over algebraically 
closed fields with $\charp$ not very good.

In contrast, for the group $\Spn$, the classical parametrization requires
only information about the square classes in $\ratk$, which is in
turn completely answered by the same question in $\resk$ when $\resp > 2$.  
It follows
that one could apply the arguments in this paper to identify 
equivalence classes of pairs $(\facet,v)$ in $\gamma_d^{-1}(\mathcal{O})$
for any nilpotent orbit $\mathcal{O}$ of $\Spn(\ratk)$,
for $\resp > 2$ (which is the set of very good primes for $\Spn$).

\end{document}